\documentclass[reqno]{amsart}

\usepackage{xcolor,comment}
\usepackage[utf8]{inputenc}
\usepackage[OT2,T1]{fontenc}
\usepackage{enumerate}
\DeclareSymbolFont{cyrletters}{OT2}{wncyr}{m}{n}
\DeclareMathSymbol{\Sha}{\mathalpha}{cyrletters}{"58}
\usepackage{tikz-cd}

\usepackage{pifont}

\title[The spin of prime ideals and level-raising of Galois representations]{The spin of prime ideals and level-raising of even Galois representations}
\date{\today}
\usepackage[foot]{amsaddr}
\author{Marius Fischer}
\email{marius.fischer@math.au.dk}
\address{Department of Mathematics, Aarhus University, 1530-432, DK-8000 Aarhus C, Denmark}
\author{Peter Vang Uttenthal}
\email{petervang@math.au.dk}
\address{Department of Mathematics, Aarhus University, 1530-421, DK-8000 Aarhus C, Denmark}

\newcommand{\SL}{\operatorname{SL}}

\newcommand{\Gal}{\mathrm{Gal}}

\newcommand{\Q}{\mathbb{Q}}
\newcommand{\Z}{\mathbb{Z}}
\newcommand{\F}{\mathbb{F}}

\usepackage{amsthm,amsmath, mathrsfs, mathtools}
\usepackage{amsfonts}
\usepackage{amssymb}
\usepackage{fancyhdr}
\usepackage{IEEEtrantools}
\usepackage{tikz-cd}
\usepackage[english]{babel}
\usepackage[utf8]{inputenc}

\usepackage{csquotes}

\newtheorem{theorem}{Theorem}
\newtheorem{lemma}[theorem]{Lemma}
\newtheorem{definition}[theorem]{Definition}
\newtheorem{proposition}[theorem]{Proposition}
\newtheorem{corollary}[theorem]{Corollary}
\newtheorem{conjecture}[theorem]{Conjecture}
\theoremstyle{definition}
\newtheorem{remark}[theorem]{Remark}

\numberwithin{theorem}{section}
\numberwithin{equation}{section}

\usepackage[hyphens,spaces,obeyspaces]{url}
\usepackage[colorlinks,allcolors=blue,hyperindex,breaklinks]{hyperref}

\hypersetup{breaklinks=true, colorlinks=true, pdfusetitle=true}

\begin{document}
\begin{abstract} 
By extending the notion of spin of prime ideals, we show that a short character sum conjecture implies that the set of primes raising the level of a certain even Galois representation has density 2/3, as conjectured by Ramakrishna in 1998. 
\end{abstract}
\maketitle
\tableofcontents

\section{Introduction}\label{introduction}
Let $G_{\mathbb{Q}}$ denote the absolute Galois group of $\mathbb{Q}$, and suppose $\rho: G_{\mathbb{Q}}\rightarrow \mathrm{GL}_2(\overline{\mathbb{Q}}_{p})$ is an irreducible $p$-adic Galois representation that is unramified outside a finite set of places. Assume further that $\rho$ is even, meaning $\det \rho(c)=1$ for a complex conjugation $c\in G_\mathbb{Q}$. The Fontaine-Mazur Conjecture \cite{FontaineMazur1995} predicts that $\rho$ can only arise from algebraic geometry if it is the Tate-twist of an even representation with finite image. In 1998, using only Galois cohomology, Ramakrishna \cite{even1} constructed the first example of a non-geometric even representation as a lift of a residual representation $\overline{\rho} : G_{\mathbb{Q} } \rightarrow \mathrm{SL}_2(\mathbb{F}_3)$ to an even surjective representation
\begin{equation}\label{even}
\rho : G_{\mathbb{Q} } \rightarrow \operatorname{SL}_2(\mathbb{Z}_3)
\end{equation}
ramified only at $3$ and $349$. Subsequently,  \cite{even2} gave a criterion on a prime $p$
for there to exist a unique surjective lift $\rho^{(p)}: G_{\mathbb{Q}} \rightarrow \mathrm{SL}_2(\mathbb{Z}_3)$ of $\overline{\rho}$
raising the level of $\rho$, cf. Section \ref{proofs_for_even_representations}.
Proving that there are infinitely many primes raising the level of \eqref{even} would give the first non-trivial infinite family of even representations onto $\mathrm{SL}_2(\mathbb{Z}_3)$ with at most three places in the level. Moreover, it would provide a counterpart to Ribet's work \cite{Ribet} on level-raising of modular Galois representations. 
\begin{comment}
Note that \cite[Corollary 1b]{ramakrishna2002} constructed infinitely many even representations onto $\mathrm{SL}_2(\mathbb{Z}_7)$ ramified at finitely many primes, 
but had no control 
over the number of additional ramified primes needed. In contrast, if $p$ raises the level of $\rho$ then $\rho^{(p)}$ has exactly three places in the level,
so it remains attractive to study level-raising of \eqref{even}. 
\end{comment}
The level-raising criterion at $p$ was found by 
prescribing a local shape of the new representation at $p$ for all $p$ in the subset 
\begin{equation}
\mathcal{C} := \{ p \equiv 1 \bmod 3:  \overline{\rho}(\operatorname{Frob}_p) \text{ has order $3$} \}.
\end{equation}
After giving a heuristic argument that the density in $\mathcal{C}$ of primes raising the level should be $2/3$, \cite[p. 99]{even2} states that
\emph{we do not even know if this happens infinitely often}. The main difficulty is that the criterion for level-raising is a splitting condition on $p$ in a number field that depends on $p$ itself, so the Chebotarev density theorem does not apply.\\

\noindent
Conditionally, we prove that the set of level-raising primes indeed has density $2/3$ in $\mathcal{C}$. Our proof is based on the spin of prime ideals first introduced by Friedlander, Iwaniec, Mazur and Rubin \cite{FIMR}. As in their work and in many other spin problems, we must assume a conjecture on short character sums. For each integer $n$, we state a Conjecture $C_n$ similar to \cite[p. 738, Conjecture $C_n$]{FIMR}, but adapted from quadratic characters to cubic characters in the obvious way. If $\chi$ is a non-principal cubic Dirichlet character of modulus $q$, our Conjecture $C_n$ stipulates a power saving in any incomplete character sum of $\chi$ over an interval of length $q^{1/n}$ (see Section \ref{short_character_sums} for further details). Our main result is the following.
\begin{theorem}\label{main_theorem}
Assume Conjecture $C_{12}$. Then the set of primes raising the level of $\rho$ has density $2/3$ in $\mathcal{C}$, i.e. 
  \begin{equation*}
\lim_{X \rightarrow \infty } \frac{\# \left\{ p \in \mathcal{C} \,:\, \textrm{$p \leq X$ and $p$ raises the level of $\rho$} \right\}}{\# \left\{ p \in \mathcal{C}\,:\,p \leq X \right\}} = \frac{2}{3}.
\end{equation*}
\end{theorem}
\noindent The above theorem is the first application of spin to a problem from the deformation theory of Galois representations, and we believe that there are similar problems where our arguments can be applied.
Note that Ribet's results in the odd case 
do not give information on how many representations raise the level of a given modular representation. In contrast, whenever a prime $p$ can be added to the level of $\rho$ in
Theorem \ref{main_theorem}, 
the new representation 
$\rho^{(p)}$ is unique.\\

\noindent
We now outline the proof of Theorem \ref{main_theorem}. 
The field fixed by the kernel of the projectivization of $\overline{\rho}$
is a totally real $A_4$ extension $K/\Q$ ramified only at $\ell=349$.
For $p \in \mathcal{C}$, let $K^{(p)}$ denote the maximal $3$-elementary extension of $K$ unramified outside $3$ and $p$. Then there is a subset $\mathcal{C}_0\subset\mathcal{C}$ of density zero such that for all $p\in \mathcal{C}\setminus\mathcal{C}_0$, 
\begin{center}
\emph{$p$ raises the level of $\rho$ if and only if $p$ has inertial degree $9$ in $K^{(p)}$.}
\end{center}
%The first step is to define a spin symbol that governs the level-raising condition. 
The first step in our proof is to show that this condition is governed by a \emph{spin symbol}. Let $F$ denote a quartic subfield of $K$ and $\zeta_3$ a primitive $3$\textsuperscript{rd} root of unity. In Section \ref{spin_symbol_section}, we define for a class of integral ideals $\mathfrak{a}$ of $F(\zeta_3)$ a spin symbol $s_{\mathfrak{a}}$ valued in $\left\{1,\zeta_3 , \zeta_3^2 \right\}$. Let $(\frac{\cdot}{\cdot})_{3,F(\zeta_3)}$ denote the cubic residue symbol over $F(\zeta_3)$. Then the results of Section \ref{spin_symbol_section} can be summarized as follows:

\begin{theorem}\label{spin_symbol_theorem}
There is a modulus $\mathbf{m}$ of $F(\zeta_3)$ and a subgroup $H_0$ of the ray-class group of $\mathbf{m}$ such that if $\mathfrak{a} \in H_0$, and we set
  \begin{equation*}
s_{\mathfrak{a}} := \left( \frac{N_{K/F}(\sigma(\alpha))}{\mathfrak{a}} \right)_{3,F(\zeta_3)}
\end{equation*}
where $\alpha$ is a generator of $N_{F(\zeta_3)/F}(\mathfrak{a})$, and $\sigma \in \mathrm{Gal}(K/\mathbb{Q})-\mathrm{Gal}(K/F)$, then $s_{\mathfrak{a}}$ is independent of the choice of $\alpha$ and $\sigma$. If $p \in \mathcal{C}$ is coprime to $\mathbf{m}$, then $p$ has degree $1$ prime factor $\mathfrak{p}$ in $F(\zeta_3)$ that is inert in $K(\zeta_3)$ and lies in $H_0$. Moreover for $p\in\mathcal{C}\setminus\mathcal{C}_0$, $p$ raises the level of $\rho$ if and only if $s_{\mathfrak{p}}\neq 1$.
\end{theorem}

\noindent
The proof uses Artin reciprocity and other tools from class field theory. Following \cite{FIMR}, we now use a sieve \cite[Proposition 5.2]{FIMR} to prove that $s_{\mathfrak{p}}$ oscillates as $\mathfrak{p}$ ranges over the degree $1$ prime ideals over the primes in $\mathcal{C}$, and we must assume Conjecture $C_{12}$ in order to succeed. The outcome is the following theorem which, together with Theorem \ref{spin_symbol_theorem}, immediately implies Theorem \ref{main_theorem}.

\begin{theorem}\label{spin_main_theorem}
Assume Conjecture $C_{12}$. Then there exists $\delta>0$ such that
\begin{equation*}
\sum_{N_{F(\zeta_3)/\mathbb{Q}}(\mathfrak{p}) \leq X} s_{\mathfrak{p}} \ll X^{1-\delta}
\end{equation*}
where the sum is taken over all prime ideals $\mathfrak{p}$ that have degree $1$ over $\mathbb{Q}$ and are inert in $K(\zeta_3)$. The same estimate is true if $\mathfrak{p}$ is restricted to an abelian Chebotarev class of $F(\zeta_3)$ contained within the set of primes that are inert in $K(\zeta_3)$.  
\end{theorem}

\noindent
Our work is the first application of the spin technique to an extension that is not Galois over $\mathbb{Q}$, and this setting causes new difficulties throughout the paper. We define the spin symbol over $F(\zeta_3)$ which is a degree $8$ extension of $\mathbb{Q}$ that has only one non-trivial automorphism, and, based on previous papers on spin, it is not clear how the spin symbol should be defined in this context. In Section \ref{spin_symbol_section}, we explain why we are forced to work over $F(\zeta_3)$ rather than its Galois closure $K(\zeta_3)$. Another challenge is that a certain lattice point counting argument first introduced in \cite{FIMR} and later improved in \cite{Koymans} breaks down. In Section \ref{counting_ideals}, we use new ideas to further improve this argument, and the results of that section can be of independent interest.\\

\noindent 
From our main results, we deduce a corollary that is in the same spirit as the initial application of spin to Selmer groups of elliptic curves \cite[Theorem 10.1]{FIMR}.
For a finite set of places $S$, let $G_S$ be the Galois group of the maximal extension of $\mathbb{Q}$ unramified outside $S$. Let $\mathrm{Ad}^0(\overline{\rho})$ be the adjoint representation of $\overline{\rho}$. In Section \ref{proofs_for_even_representations}, we define the Selmer group $H^1_{\mathcal{N}}(G_S,\mathrm{Ad}^0(\overline{\rho}))$, and we have the following result.
\begin{corollary} \label{selmerdensity} Let $S= \{3,349\}$. Then we have
    \begin{align*}
        \dim H^1_{\mathcal{N}}(G_{S \cup \{ p\} }, \operatorname{Ad}^0(\overline{\rho}) )
        = 
        \begin{cases}
           \dim H^1_{\mathcal{N}}(G_{S}, \operatorname{Ad}^0(\overline{\rho}) ) + 1 & \textrm{if $s_{\mathfrak{p}} = 1$, } \\
             \dim H^1_{\mathcal{N}}(G_{S }, \operatorname{Ad}^0(\overline{\rho}) ) & \textrm{if $s_{\mathfrak{p}} \neq 1$.}  
        \end{cases}
    \end{align*}
Assuming Conjecture $C_{12}$, the Selmer-rank increases by $1$ one-third of the time and remains the same two-thirds of the time.
\end{corollary}

\noindent
Increasing the ranks of Selmer groups by allowing ramification at just one additional prime is generally considered a difficult problem. Instead, it has become more common to relax the conditions and allow ramification at two primes \cite{KLR1, KLR2, FKP} so the above corollary is an unusually strong result.

\subsection{Acknowledgments}
We are grateful to Peter Koymans, Simon Kristensen, Paul Nelson,  and Ravi Ramakrishna for helpful conversations. 
This work is supported by grant VIL54509 from Villum Fonden.

\section{Notation}
In this section, we explain our notation for concepts related to number fields and class field theory. If $E$ is a number field, we write $\mathcal{O}_E$ for its ring of integers and $\mathcal{O}_E^{\times}$ for the unit group of $\mathcal{O}_E$. Suppose now that $M/E$ is a finite extension. If $\mathfrak{p}$ is non-zero prime ideal of $\mathcal{O}_E$, and $\mathfrak{P}$ is prime over $\mathfrak{p}$ in $M$, we write $f_{\mathfrak{P}/\mathfrak{p}}$ and $e_{\mathfrak{P}/\mathfrak{p}}$ for the inertial and ramification degree respectively. If $M/E$ is Galois, then $f_{\mathfrak{P}/\mathfrak{p}}$ and $e_{\mathfrak{P}/\mathfrak{p}}$ do not depend on the overlying prime $\mathfrak{P}$, and we write $f(\mathfrak{p},M/E)$ and $e(\mathfrak{p},M/E)$ instead. If $\mathfrak{a}$ is a non-zero fractional ideal of $M$, we write $N_{M/E}(\mathfrak{a})$ for its norm onto $E$.\\

\noindent
A modulus $\mathbf{m}$ of $E$ is by definition a pair $(\mathbf{m}_0 , \mathbf{m} _\infty)$ where $\mathbf{m}_0$ is a non-zero ideal of $\mathcal{O}_E $, and $\mathbf{m} _\infty$ is set of real embeddings of $E$. If $\mathbf{m} _\infty = \emptyset$ (e.g. if $E$ is totally complex), we use $\mathbf{m}$ and $\mathbf{m}_0$ interchangeably. When $\mathbf{m}$ and $\mathbf{m} '$ are moduli of $E$, we say that $\mathbf{m}$ divides $\mathbf{m} '$ if $\mathbf{m}_0 \mid  \mathbf{m}_0  '$ as ideals of $\mathcal{O}_E$, and $\mathbf{m} _\infty \subset \mathbf{m} _\infty '$. We use the following notation:
\begin{itemize}
\item $I_E(\mathbf{m})$ denotes the group of non-zero fractional ideals of $E$ coprime to $\mathbf{m}_0$.
\item $P_{\mathbf{m}}:= \left\{ (\alpha) \in I_E( \mathbf{m})\,:\, \alpha \equiv 1 \pmod{\mathbf{m}_0}\textrm{ and }\sigma(\alpha)>0\textrm{ for all }\sigma \in \mathbf{m} _\infty \right\}.$ 
\item $H_E(\mathbf{m}):= I_K(\mathbf{m})/P_{\mathbf{m}}$ denotes the ray class group of $\mathbf{m}$.
\item $E(\mathbf{m})$ denotes the ray class field of $\mathbf{m}$.
\end{itemize}
If $M$ is a finite abelian extension of $E$, we also use the following notation:
\begin{itemize}
\item $\mathfrak{f}(M/E)$ denotes the conductor of the extension $L/K$.
\item If $\mathbf{m}$ is a modulus divisible by all primes of $E$ that ramify in $E$, then $\Phi_{E/K,\mathbf{m} }: I_E(\mathbf{m}) \rightarrow \Gal(M/E)$ denotes the Artin map.
\end{itemize}
By Artin reciprocity, $\Phi_{M/E, \mathbf{m}}$ is surjective, and its kernel contains $P_{\mathbf{m}}$ if and only if $\mathfrak{f}(M/E)\mid \mathbf{m}$.

\section{The cubic residue symbol}
Before defining the spin symbol, we recall the definition of the cubic residue symbol. Suppose $E$ is a number field containing $\zeta_3$, a primitive $3$\textsuperscript{rd} root of unity. Let $\mathfrak{p}$ be a prime ideal of $E$ not containing $3$. If $\alpha \in \mathcal{O}_E$, we define the cubic residue symbol $\big( \frac{\alpha}{\mathfrak{p}}\big)_{3, E}$ as the unique element of $\left\{ 1, \zeta_3 , \zeta_3^2 , 0 \right\}$ satisfying
\begin{equation}
\left( \frac{\alpha}{\mathfrak{p}} \right)_{3, E} \equiv \alpha^{\frac{N(\mathfrak{p})-1}{3}} \pmod{\mathfrak{p}}
\end{equation}
where $N_{E/\mathbb{Q}}(\mathfrak{p}):= \# \mathcal{O}_E / \mathfrak{p}$ is the absolute norm of $\mathfrak{p}$. If $\mathfrak{a}$ is a non-zero integral ideal of $E$ not containing $3$ that factors into prime ideals as $\prod_{i=1}^r \mathfrak{p}_i^{a_i}$, we define
\begin{equation*}
\left( \frac{\alpha}{\mathfrak{a}} \right)_{3,E}:= \prod_{i=1}^r \left( \frac{\alpha}{\mathfrak{p}_i} \right)_{3,E}^{a_i}.
\end{equation*}
Clearly, this expression only depends on the residue class of $\alpha$ modulo $\mathfrak{a}$. We will need the following version of cubic reciprocity:

\begin{proposition}\label{cubic_reciprocity}
Let $\alpha , \beta \in \mathcal{O}_E$ with $\beta$ coprime to $3$. Then $\big(\frac{\alpha}{\beta}\big)_{3,E}$ only depends on the residue class of $\beta$ modulo $27 \alpha$. If $\alpha$ is also coprime to $3$, we have
  \begin{equation*}
\left( \frac{\alpha}{\beta} \right)_{3,E}= \mu \left( \frac{\beta}{\alpha} \right)_{3,E}
\end{equation*}
for some $\mu \in \left\{ 1, \zeta_3 , \zeta_3^2 \right\}$ only depending on the values of $\alpha$ and $\beta$ modulo $27$. 
\end{proposition}

\noindent
The reader might have noticed that $(\frac{\alpha}{\beta})_{3,E}$ only depends on the ideal generated by $\beta$. On the other hand, the value of $\beta$ modulo $27\alpha$ can change if we multiply $\beta$ by a unit, but it is implicitly part of the statement of the proposition that it does not change in a way that affects the residue symbol. The proof uses the product formula for Hilbert symbols, and we are grateful to Peter Koymans for explaining it to us.

\begin{proof}
If $\alpha$ and $\beta$ are not coprime, then we can read it off from the residue class of $\beta$ modulo $27 \alpha$, and in this case the cubic residue symbol equals $0$. Hence we may assume that $\alpha$ and $\beta$ are coprime. To prove the proposition, we write the cubic residue symbol in terms of local Hilbert symbols. Suppose $v$ is a place of $E$ (finite or infinite), and let $E_v$ denote the completion of $E$ with respect to $v$. Let 
\begin{equation*}
\left( \frac{\cdot\,,\,\cdot}{v} \right): E_{v} \times E_{v} \rightarrow \left\{ 1, \zeta_3 , \zeta_3^2 \right\}
\end{equation*}
denote the cubic Hilbert symbol in $E_v$ (see \cite[Ch. VI, \S 8]{neukirch} for a definition). Since $E$ contains $\zeta_3$, all infinite places of $E$ are complex, and the corresponding Hilbert symbols are trivial (this fact is clear from the definition given in \cite{neukirch}). If $\mathfrak{p}$ is a finite place of $E$ not dividing $3$, the Hilbert symbol is related to cubic residue symbol via
\begin{equation*}
\left( \frac{\alpha}{\mathfrak{p}} \right)_{3,E} = \left( \frac{\pi_{\mathfrak{p}} , \alpha}{\mathfrak{p}} \right)
\end{equation*}
where $\pi_{\mathfrak{p}}$ is any uniformizer in $E_{\mathfrak{p}}$ \cite[p. 415]{neukirch}. Moreover, if $\mathfrak{p}\nmid 3$, and $u_1$ and $u_2$ are units in the ring of integers in $E_{\mathfrak{p}}$, then $( \frac{u_1 , u_2}{\mathfrak{p}})=1$. This fact follows from \cite[Ch. V, Proposition 3.2(iii), Lemma 3.3 and Corollary 1.2]{neukirch}. Combined with the product formula for the Hilbert symbols \cite[Ch. VI, Theorem 8.1]{neukirch}, we get
\begin{equation*}
\left( \frac{\alpha}{\beta} \right)_{3,E}= \prod_{\mathfrak{p} \mid \beta} \left( \frac{\beta , \alpha}{\mathfrak{p}} \right) = \prod_{\mathfrak{p} \mid 3\alpha} \left( \frac{\beta,\alpha}{\mathfrak{p}} \right)^{-1} = \prod_{\mathfrak{p} \mid 3 \alpha} \left( \frac{\alpha, \beta}{\mathfrak{p}} \right)
\end{equation*}
where we have used that $\beta$ is coprime to $3$ and $\alpha$, and that swapping the arguments inverts the Hilbert symbol. We can write the last expression as
\begin{equation}\label{hilbert_symbols}
\prod_{\mathfrak{p} \mid 3} \left( \frac{\alpha, \beta}{\mathfrak{p}} \right) \prod_{\substack{
\mathfrak{p} \mid \alpha \\\mathfrak{p} \nmid 3 
}} \left( \frac{\alpha , \beta}{\mathfrak{p}} \right).
\end{equation}
If $\mathfrak{p} \mid 3$, it follows by Hensel's lemma that any element in the ring of integers of $E_{\mathfrak{p}}$ that is $1$ modulo $27$ is a cube so the first factor only depends on $\alpha$ and $\beta$ modulo $27$. If $\mathfrak{p} \mid \alpha$, and $\mathfrak{p} \nmid 3$, let $\pi_{\mathfrak{p}}$ denote a uniformiser, and write $\alpha = u \pi_{\mathfrak{p}}$ for some $\mathfrak{p}$-adic unit $u$ and integer $n$. Then
\begin{equation*}
\left( \frac{\alpha , \beta}{\mathfrak{p}} \right)= \left( \frac{u, \beta}{\mathfrak{p}} \right) \left( \frac{\pi_{\mathfrak{p}},\beta }{\mathfrak{p}} \right)^n= \left( \frac{\beta}{\mathfrak{p}} \right)_{3,E}^n
\end{equation*}
where we have used that both $u$ and $\beta$ are $\mathfrak{p}$-adic units. The last expression only depends on $\beta$ modulo $\mathfrak{p}$ and hence only on $\beta$ modulo $\alpha$. This proves the first part of the proposition. For the second part, we assume that $\alpha$ is coprime to $3$. Then second factor in \eqref{hilbert_symbols} equals $(\frac{\beta}{\alpha})_{3,E}$. Hence
\begin{equation*}
\left( \frac{\alpha}{\beta} \right)_{3,E} = \prod_{\mathfrak{p} \mid 3} \left( \frac{\alpha, \beta}{\mathfrak{p}} \right) \left( \frac{\beta}{\alpha} \right)_{3,E},
\end{equation*}
and we have already explained why the product over the places dividing $3$ only depends on $\alpha$ and $\beta$ modulo $27$ so the proof is complete.
\end{proof}

\noindent
We also need the following lemma which explains how to pass between cubic residue symbols in a Galois extension. It will allow us to do computations with the spin symbol in the Galois closure $K(\zeta_3)$ of $F(\zeta_3)$. The lemma can readily be generalized to any power-residue symbol. 

\begin{lemma}\label{alternative}
Let $\mathbb{Q}(\zeta_3)\subset E_0\subset E$ be number fields such that $E/E_0$ is a Galois extension. Suppose $\mathfrak{a}$ is an ideal of $\mathcal{O}_{E_0}$ coprime to $3 \Delta(E / E_0)$, and $\beta \in \mathcal{O}_{E}$. Then
  \begin{equation*}
\left( \frac{N_{E/E_0}(\beta)}{\mathfrak{a}} \right)_{3,E_0} = \left( \frac{ \beta}{\mathfrak{a} \mathcal{O}_{E}} \right)_{3, E}.
\end{equation*}
\end{lemma}

\begin{proof}
It is enough to consider the case when $\mathfrak{a} = \mathfrak{p}$ is a prime ideal of $\mathcal{O}_{E_0}$. Let $G:= \mathrm{Gal}(E / E_0)$. Fix a prime ideal $\mathfrak{P}$ of $\mathcal{O}_{E}$ lying over $\mathfrak{p}$, and let $D_{\mathfrak{P}/\mathfrak{p} } \leq G$ denote the corresponding decomposition group so that $\mathfrak{p} \mathcal{O}_{E} = \prod_{ \sigma \in G/D_{\mathfrak{P}/\mathfrak{p} }}\sigma(\mathfrak{P})$. Since $G$ fixes $\mathbb{Q}(\zeta_3)$, we have 
\begin{equation*}
\left( \frac{\beta}{\mathfrak{p} \mathcal{O}_{E}} \right)_{3,E} = \prod_{\sigma \in G/ D_{\mathfrak{P}/\mathfrak{p} }} \left( \frac{\sigma^{-1} (\beta)}{\mathfrak{P} } \right)_{3,E} \equiv \prod_{\sigma \in G/ D_{\mathfrak{P}/\mathfrak{p}}} \sigma^{-1}(\beta)^{\frac{N_{E/\mathbb{Q}}(\mathfrak{P})-1}{3}} \pmod{\mathfrak{P}}.
\end{equation*}
Since $\mathfrak{p}$ does not divide $\Delta(E/ E_0)$, $\mathfrak{p}$ is unramified in $E $ so $D_{\mathfrak{P}/\mathfrak{p} }$ is cyclic and generated by a Frobenius element $\tau$. If $q := N_{E_0 / \mathbb{Q}}(\mathfrak{p})$, we have $N(\mathfrak{P} )=q^{f_{\mathfrak{P}/ \mathfrak{p}  }}$, and 
\begin{equation*}
\begin{split}
  \prod_{\sigma \in G / D_{\mathfrak{P}/ \mathfrak{p} }} \sigma^{-1}(\beta)^{\frac{N_{E / \mathbb{Q}}(\mathfrak{P})-1}{3}} & = \left[ \prod_{\sigma \in G/ D_{\mathfrak{P}/\mathfrak{p}}} \prod_{i=0}^{f_{\mathfrak{P}/\mathfrak{p} }-1} \sigma^{-1}(\beta)^{q^i} \right]^{\frac{q-1}{3}}\\
& \equiv \left[ \prod_{\sigma \in G/ D_{\mathfrak{P}/ \mathfrak{p} }} \prod_{i=0}^{f_{\mathfrak{P}/ \mathfrak{p} }-1} \tau^i \sigma^{-1}(\beta) \right]^{\frac{q-1}{3}} \pmod{\mathfrak{P}}\\
  &= N_{E/  E_0}(\beta)^{\frac{q-1}{3}} 
\end{split}
\end{equation*}
since when $\sigma$ traverses a set of representatives for $G/D_{\mathfrak{P}/\mathfrak{p} }$, $\sigma^{-1}$ traverses a set of representatives for $D_{\mathfrak{P}/\mathfrak{p}}\backslash G$. It follows that
\begin{equation*}
\left( \frac{\beta}{\mathfrak{p} \mathcal{O}_{E}} \right)_{3,E} \equiv  \left( \frac{N_{E/ E_0}(\beta)}{\mathfrak{p}} \right)_{3,E_0} \pmod{\mathfrak{P} },
\end{equation*}
and since both sides are valued in $\left\{ 1,\zeta_3 ,\zeta_3^2,0 \right\}$, and $3 \notin \mathfrak{P} $, they must be equal.
\end{proof}

\section{The spin symbol}\label{spin_symbol_section}
In this section, we elaborate on the construction of the spin symbol defined in Theorem \ref{spin_symbol_theorem} and explain why it captures the level-raising condition. Recall from Section \ref{introduction} that $K$ is a totally real $A_4$-extension only ramified at $349$. It can be realized as the splitting field of the quartic polynomial
\begin{equation*}
n(x):=x^4-x^3-10x^2+3x+20
\end{equation*}
of discriminant $349^2$ \cite[p. 567]{even1}. Moreover, $F$ denotes a quartic subfield of $K$, or equivalently the field obtained by adjoining a single root of $n(x)$ to $\mathbb{Q}$. We defined $\mathcal{C}$ as set of rational primes that are $1$ modulo $3$, are unramified in $K$ and have inertial degree $3$ in $K$. For $p\in \mathcal{C}$, $K^{(p)}$ denotes the maximal $3$-elementary extension of $K$ unramified outside $3p$, and for all $p\in\mathcal{C}$ outside a set of density zero, level-raising is equivalent to $f(p,K^{(p)}/\mathbb{Q})=9$.\\

\noindent
Our spin symbol will be defined over $F(\zeta_3)$ which is not a Galois extension of $\mathbb{Q}$. As mentioned previously, this causes many new challenges, and we now explain why we are forced to work over $F(\zeta_3)$ rather than its Galois closure $K(\zeta_3)$. The extension $K^{(p)}/K$ is $3$-elementary so we must work with a cubic spin symbol defined over a field containing $\zeta_3$. The natural choice is therefore $K(\zeta_3)$, but this causes a major problem: All primes of $\mathcal{C}$ have degree $3$ in $K(\zeta_3)$ so in order to get an estimate as in Theorem \ref{spin_main_theorem}, we must find cancellation in a sum over degree $3$ prime ideals. Given $X$, the number of prime ideals of degree $3$ over $\mathbb{Q}$ and norm at most $X$ is bounded by a constant times $X^{1/3}$ so this would be a hopeless task, even if we assume GRH because this only predicts an error term of size $X^{1/2}\log X$ in the prime number theorem for number fields.\\

\noindent
To circumvent this problem, we use the observation from \cite{even2} that the condition $f(p, K^{(p)}/\mathbb{Q})=9$ can lowered to the quartic subfield $F$. This is obtained by adjoining a single root of $n(x)$ to $\mathbb{Q}$. If $p \in \mathcal{C}$, then $p$ is unramified in $K$ and has inertial degree $3$ in $K$. Hence, any Frobenius element over $p$ in $\Gal(K/\mathbb{Q}) \simeq A_4$ has order $3$ and must act on the roots of $n(x)$ as a $3$-cycle. It follows that $p$ factors in $F$ as $\mathfrak{p}_1 \mathfrak{p}_2$ where $f_{\mathfrak{p}_1/p}=3$, and $f_{\mathfrak{p}_2/p}=1$. Moreover, the factorization of $n(x)$ modulo $3$ is $(x^3 + x^2 + x + 2)(x+1)$ so there is a similar factorization of $3$ in $F$ as $ 3_1 3_2  $ where $f_{3_1/3}=3$, and $f_{3_2 / 3}=1$. We then have the following result:

\begin{proposition}\label{field_lowering}
Let $p \in \mathcal{C}$, and let $F^{(\mathfrak{p}_2)}$ denote maximal abelian $3$-elementary extension of $F$ unramified away from $3_1$ and $\mathfrak{p}_2$. Then $\Gal(F^{(\mathfrak{p}_2)}/F)\simeq \mathbb{Z} / 3 \mathbb{Z}$, and $f(p , K^{(p)}/\mathbb{Q})=9$ if and only if $f(\mathfrak{p}_1 , F^{(\mathfrak{p}_2)}/F)=3$. 
\end{proposition}

\begin{proof}
The first claim follows from Theorem A and Lemma 1 in \cite{even2}. The implication $f(\mathfrak{p}_1 , F^{(\mathfrak{p}_2)}/F)=3 \Rightarrow f(p, K^{(p)}/\mathbb{Q})=9$ is Proposition 4 \cite{even2}, and the proof readily upgrades to a biimplication.
\end{proof}

\noindent
If $p \in \mathcal{C}$ factors as $\mathfrak{p}_1 \mathfrak{p}_2$ in $F$ as above, then $\mathfrak{p}_1$ and $\mathfrak{p}_2$ split completely in $F(\zeta_3)$ since $p \equiv 1 \pmod{3}$. Hence $p$ has two prime factors of degree $1$ in $F(\zeta_3)$ which are inert in $K(\zeta_3)$. Conversely, if $\mathfrak{p}$ is a prime ideal of $F(\zeta_3)$ of degree $1$ over $\mathbb{Q}$ and inert in $K(\zeta_3)$ then $\mathfrak{p}$ lies over a prime in $\mathcal{C}$. 
For certain integral ideals $\mathfrak{a}$ of $F(\zeta_3)$, we then define a spin symbol $s_{\mathfrak{a}}$ such that when $\mathfrak{p}$ has degree $1$ over $\mathbb{Q}$ and is inert in $K(\zeta_3)$, then $s_{\mathfrak{p}}\neq 1$ if and only if $f(p, K^{(p)}/\mathbb{Q})=9$ where $(p)=\mathfrak{p} \cap \mathbb{Q}$. Since $\mathfrak{p}$ has degree $1$ over $\mathbb{Q}$, there is now hope that Theorem \ref{main_theorem} can be proved.

To make the spin symbol $s_{\mathfrak{a}}$ well-defined, we only consider integral ideals $\mathfrak{a}$ in $\mathcal{O}_{F(\zeta_3)}$ such that all units $v \in \mathcal{O}_{F}^\times$ satisfy $(\frac{v}{\mathfrak{a}})_{3,F(\zeta_3)}=1$. Since $\mathcal{O}_{F}^\times$ is finitely generated, this will impose a finite number of congruence conditions on $\mathfrak{a}$, and we end up with a spin symbol defined on a subgroup $H_0$ of a certain ray class group $H_{F(\zeta_3)}(\mathbf{m})$ of $F(\zeta_3)$. We then verify that all prime ideals of interest to us lie in $H_0$ (possibly with a finite number of exceptions). After having defined the spin symbol, we use Artin reciprocity and other tools from class field theory to verify that it correctly encodes the level-raising condition.\\

\noindent
To define the ray class group $H_{F(\zeta_3)}(\mathbf{m})$ and the subgroup $H_0$, we introduce some notation. We abbreviate $F(\zeta_3)$ by $F'$ and $K(\zeta_3)$ by $K'$. We also introduce the following extension of $K'$:
\begin{equation*}
M := K'\left(\sqrt[3]{\mathcal{O}_K^{\times}}\right)
\end{equation*}
meaning that $M$ is the field obtained by adjoining the cube roots of a system of fundamental units in $K$. In the context of the tame Gras-Munnier theorem, this is known as the $3$-governing field of $K$.

To define $H_{F'}(\mathbf{m})$, we must specify the modulus $\mathbf{m}$. Let $v_1, v_2, v_3$ be a system of fundamental units for $\mathcal{O}_F^\times$, and set $F_i ' := F '(\sqrt[3]{v_i})$ for $i=1,2,3$. Since $v_1 , v_2 , v_3$ are fundamental units, these extensions are non-trival, and because $\zeta_3 \in F'$, they are cyclic of degree $3$. We now take $\mathbf{m}$ to be any modulus of $F'$ satisfying the following conditions:

\begin{enumerate}
\item $\mathbf{m}$ is divisible by all primes of $\mathbb{Q}$ that ramify in the governing field $M$;
\item $\mathbf{m}$ is divisible by $\mathfrak{f}(F_i '/F')$ for $i=1,2,3$;
\item $\mathbf{m}$ is divisible by $\mathfrak{f}(K'/F')$.
\end{enumerate}
Condition (1) implies that $\mathbf{m}$ is divisible by $3$ and by $\ell = 349$. The following lemma ensures that our spin symbol $s_{\mathfrak{a}}$ will be well-defined when $\mathfrak{a}$ lies in certain subgroup of $H_{F'}(\mathbf{m})$.

\begin{lemma}\label{units_are_cubes}
There is a subgroup $H_0$ of $H_{F'}(\mathbf{m})$ such that for $\mathfrak{a} \in H_0$, we have $\big(\frac{v}{\mathfrak{a}}\big)_{3, F'}=1$ for all $v \in \mathcal{O}_F^\times$.
\end{lemma}

\begin{proof}
It is enough to ensure that $\big(\frac{v_i}{\mathfrak{a}}\big)_{3,F'}=1$ for $i=1,2,3$ when $\mathfrak{a} \in H_0$. For each $i\in\{1,2,3\}$, we have a commutative diagram
\begin{equation}\label{diagram}
\begin{tikzcd}
	{I_{F'}(\mathbf{m})} && {\Gal(F_i'/F')} \\
	\\
	&& {\{1,\zeta_3,\zeta_3^2\}}
	\arrow["{\Phi_{F_i'/F',\mathbf{m}}}", from=1-1, to=1-3]
	\arrow["{\big(\frac{v _i}{\bullet}\big)_{3,F'}}"', from=1-1, to=3-3]
	\arrow["\wr", from=1-3, to=3-3]
\end{tikzcd}
\end{equation}
where the vertical map is the isomorphism $\sigma \mapsto \sigma (\sqrt[3]{v_i})/\sqrt[3]{v_i}$, see \cite[p. 166]{milne}. Thus if  $ \mathfrak{a} \in I_0 := \cap_{i=1}^3 \ker \Phi_{F_i'/F', \mathbf{m}}$, we have $\big(\frac{v}{ \mathfrak{a}}\big)_{3,F'}=1$ for all $v \in \mathcal{O}_F^\times$. Since $\mathbf{m}$ is divisible by $f(F_i '/F')$ for $i=1,2,3$, $P_{\mathbf{m}} \subset \ker \Phi_{F_i'/F', \mathbf{m}}$ for all $i$, and hence $P_{\mathbf{m}}\subset I_0$. We then take $H_0 := I_0 / P_{\mathbf{m}}$. 
\end{proof}

The next lemma shows that the $H_0$ contains all but finitely many of the prime ideals of interest to us. 

\begin{lemma}\label{units_are_cubes_modulo_primes}
Let $\mathfrak{p}$ be a prime ideal of $F'$ such that $\mathfrak{p}$ has degree $1$ over $\mathbb{Q}$, and $\mathfrak{p}$ is inert in $K(\zeta_3)$. If $\mathfrak{p} \nmid \mathbf{m}$, then $\mathfrak{p} \in H_0$. 
\end{lemma}

\begin{proof}
By definition of $H_0$, we must show that $\mathfrak{p} \in \ker \Phi_{F_i '/F', \mathbf{m}}$ for $i =1,2,3$ which, by the diagram in (\ref{diagram}), is equivalent to all units $v \in \mathcal{O}_F^\times$ being a cube modulo $\mathfrak{p}$. Let $p$ be the prime of $\mathbb{Q}$ lying under $\mathfrak{p}$. Then $f(p, K(\zeta_3)/\mathbb{Q})=3$ which forces $p \equiv 1 \pmod{3}$, and $f(p,K/\mathbb{Q})=3$. Hence $p \mathcal{O}_F = \mathfrak{p}_1 \mathfrak{p}_2$ where $f_{\mathfrak{p}_1/p}=3$, and $f_{\mathfrak{p}_2 /p}=1$. Since $f_{\mathfrak{p} / \mathfrak{p}_2}=1$, $v \in \mathcal{O}_F^\times$ is a cube modulo $\mathfrak{p}$ if and only if $v$ is a cube modulo $\mathfrak{p}_2$. This is equivalent to $f(\mathfrak{p}_2 , F(\sqrt[3]{v})/F)=1$ (since $p \equiv 1 \pmod{3}$, it does not matter which cube root we choose).

We can assume that $v$ is not already a cube in $F$, and in particular, $v \notin \mathbb{Q}$. Since there are no intermediate fields strictly between $F$ and $\mathbb{Q}$ (as $A_4$ has no subgroup of index $2$), we have $F = \mathbb{Q}(v)$, and $F(\sqrt[3]{v})=\mathbb{Q}(\sqrt[3]{v})$. Let $v_1, v_2 , v_3 , v_4$ denote the Galois conjugates of $v$ enumerated such that $v_1 =v$. Since $p$ factors as the product of a degree $1$ and a degree 3 prime ideal in $\mathcal{O}_F$ and is unramified in $M$ (as $\mathfrak{p}\nmid \mathbf{m}$), we can choose a Frobenius element $\sigma_p \in \Gal(M/\mathbb{Q})$ over $p$ in the governing field $M$ such that $\sigma_p(v_1 )=v_1$, and $\sigma_p$ cyclically permutes $v_2 , v_3 , v_4$. The primes over $p$ in $\mathbb{Q}(\sqrt[3]{v})$ are in bijection with orbits of the Galois conjugates of $\sqrt[3]{v}$ under the action $\sigma_p$, and orbit sizes correspond to inertial degrees. We show that $\sigma_p$ pointwise fixes the three cube roots of $v$ in $M$. Then $p$ has three degree $1$ factors in $\mathbb{Q}(\sqrt[3]{v})$, and since $f_{\mathfrak{p}_1/p }=3$ these must lie over $\mathfrak{p}_2$, i.e. $f(\mathfrak{p}_2 , F(\sqrt[3]{v})/F)=1$ as desired.

To explain why this is the case, we fix a cube root $\sqrt[3]{v_2}$ of $v_2$. Then, in some order, $\sqrt[3]{v_2 }, \sigma_p (\sqrt[3]{v_2}), \sigma_p^2(\sqrt[3]{v_2})$ are cube roots of $v_2, v_3 , v_4$. Since $v_1$ is a unit in $\mathcal{O}_F$, $v_1 v_2 v_3 v_4 = N_{F / \mathbb{Q}}(v_1)= \pm 1$ so one sees that $\pm \left[ \sqrt[3]{v_2} \sigma_p(\sqrt[3]{v_2}) \sigma_p^2(\sqrt[3]{v_2}) \right]^{-1}$ is a cube root of $v_1$ which is fixed by $\sigma_p$, since $\sigma_p$ has order $3$. The remaining two cube roots of $v_1$ are also fixed by $\sigma_p$ because $\zeta_3$ is fixed by $\sigma_p$. This completes the proof. 
\end{proof}

Finally, we define the spin symbol. We will make use of the fact that $F$ has class number $1$ \cite[p. 578]{even1}. Fix an automorphism $\sigma \in \Gal(K/\mathbb{Q})-\Gal(K/F)$. 

\begin{definition}\label{the_spin_symbol}
Let $\mathfrak{a} \in H_0$. Then we define the spin symbol $s_{\mathfrak{a}}$ as
\begin{equation*}
s_{\mathfrak{a}}:= \Bigg(\frac{N_{K/F}(\sigma(\alpha))}{\mathfrak{a}}\Bigg)_{3,F'}
\end{equation*}
where $\alpha$ is any generator of $N_{F'/F}(\mathfrak{a})$. 
\end{definition}

\noindent
It is an easy consequence of Lemma \ref{units_are_cubes} that the spin symbol is independent of the choice of generator $\alpha$. By Lemma \ref{canonical} below, the definition is also independent of the choice of $\sigma$. We remark that the appearance of the norm $N_{K/F}$ is unusual for a spin symbol, but it is necessary because unless $\alpha\in\mathbb{Q}$, $\sigma(\alpha)$ lands outside of $F'$. Another reason is that the level-raising condition reduces to a cubic property of a degree $1$ prime ideal relative to a degree $3$ prime ideal, c.f. the proof of Proposition \ref{spin_captures_level_raising} below.

By Lemma \ref{alternative}, it follows that the spin symbol can be lifted to $K'$ by removing the norm $N _{K/F}$:
\begin{equation}\label{alternative_expression}
 s _{\mathfrak{a}} = \left(\frac{\sigma(\alpha)}{\mathfrak{a}\mathcal{O}_{K'}}\right)_{3,K' }.
\end{equation}

\noindent
Here we have used that $N _{K/F}(\sigma(\alpha))=N_{K'/F'}(\sigma(\alpha))$ for $\alpha \in K$. We will use this expression when proving Theorem \ref{main_theorem}.\\

\noindent
The following lemma shows that the spin symbol has a more canonical expression that is clearly independent of the choice of $\sigma$. We have chosen the above definition because it makes it easier to prove that the spin symbol has the right properties.

\begin{lemma}\label{canonical}
Let $\alpha \in F$ and $\sigma \in \Gal(K/\mathbb{Q})-\Gal(K/F)$. Then
\begin{equation*}
N_{K/F}(\sigma(\alpha))= \sigma_1(\alpha) \sigma_2(\alpha) \sigma_3(\alpha)
\end{equation*}
where $\sigma_1 , \sigma_2 , \sigma_3$ are the three double-transpositions in $\Gal(K/\mathbb{Q}) \cong A_4$.
\end{lemma}

\begin{proof}
Let $\tau$ be a generator of $\Gal(K/F)\cong \mathbb{Z}/3 \mathbb{Z}$. Thus
\begin{equation*}
 N_{K/F }(\sigma(\alpha))= \sigma(\alpha) \tau \sigma(\alpha)\tau^2 \sigma(\alpha).
\end{equation*}
We claim that $\sigma \tau^i$ is a double transposition for some $i \in \left\{ 0,1,2 \right\}$. Indeed, let $V_4 \leq A_4$ be the subgroup generated by double transpositions. Then $\sigma, \sigma \tau, \sigma \tau^2$ must be distinct modulo $V_4$ because otherwise $\sigma \tau^{i-j} \sigma^{-1} \in V_4$ for some distinct $i,j \in \left\{ 0,1,2 \right\}$, but $\sigma \tau^{i-j} \sigma^{-1}$ is a $3$-cycle so that is impossible. Since $V_4$ has index $3$ in $A_4$, $\sigma \tau^i \in V_4$ for some $i$, and $\sigma \tau^i\neq 1$ because $\sigma \notin \Gal(K/F)$. The three double transpositions are now $\sigma \tau^i$, $\tau (\sigma \tau^{i}) \tau^{-1}$ and $\tau^2 (\sigma \tau^{i}) \tau^{-2}$. Since $\tau$ fixes $\alpha$, we have
\begin{equation*}
\sigma(\alpha) \tau \sigma(\alpha) \tau^2 \sigma(\alpha)=\sigma \tau^i(\alpha) \tau \sigma \tau^{i-1}(\alpha) \tau^2 \sigma \tau^{i-2}(\alpha)=\sigma_1(\alpha) \sigma_2(\alpha) \sigma_3(\alpha)
\end{equation*}
as desired. 
\end{proof}

We now explain why the spin symbol captures the level-raising condition.

\begin{proposition}\label{spin_captures_level_raising}
Let $\mathfrak{p}\nmid \mathbf{m}$ be a prime ideal of $F '$ of degree $1$ over $\mathbb{Q}$ and inert in $K'$. Suppose $\mathfrak{p}$ lies over the prime $p$ of $\mathbb{Q}$. Then $f(p, K^{(p)}/\mathbb{Q})=9$ if and only if $s_{\mathfrak{p}}\neq 1$.
\end{proposition}

\begin{proof}
We have $f(p, K /\mathbb{Q})=3$ so $p \mathcal{O}_F= \mathfrak{p}_1 \mathfrak{p}_2$ where $f_{\mathfrak{p}_1 /p}=3$ and $f_{\mathfrak{p}_2 /p}=1$. Moreover, $\mathfrak{p}_2 \mathcal{O}_{F'}= \mathfrak{p} \mathfrak{p} '$ for some $\mathfrak{p} '\neq \mathfrak{p}$ since $f_{\mathfrak{p} /p}=1$. Let $\mathfrak{p}_2 =(\pi_2)$ so that
\begin{equation*}
s_{\mathfrak{p}}= \Bigg(\frac{N_{K/F}(\sigma(\pi_2))}{\mathfrak{p}}\Bigg)_{3,F'}.
\end{equation*}
Since $\mathfrak{p}_2$ is inert in $K$, $\sigma(\pi_2) \mathcal{O}_K = \sigma(\mathfrak{p}_2 \mathcal{O}_K )$ is a prime of $K$ lying over $p$ in $\mathbb{Q}$. Because $f(p,K/\mathbb{Q})=3$, the decomposition group of $\mathfrak{p}_2 \mathcal{O}_K $ has size $3$ so it must equal $\Gal(K/F)$. We chose $\sigma \notin \Gal(K/F)$, so $\sigma(\mathfrak{p}_2 \mathcal{O}_K )$ must lie over $\mathfrak{p}_1$ in $F$. It has degree $1$ over $F$ since $\mathfrak{p}_1$ already has degree $3$ over $\mathbb{Q}$. Hence $\mathfrak{p}_1 = N_{K/F }(\sigma(\mathfrak{p}_2 \mathcal{O}_K))=(N_{K/F}(\sigma(\pi_2)))$. In other words, $\pi_1:=N_{K/F}(\sigma(\pi_2))$ is a generator for $\mathfrak{p}_1$. By Proposition \ref{field_lowering}, our task is now to show
\begin{equation*}
\left( \frac{\pi_1}{\mathfrak{p}} \right)_{3,F'}\neq 1 \Longleftrightarrow f(\mathfrak{p}_1, F^{(\mathfrak{p}_2)}/F)=3. 
\end{equation*}
Equivalently, we must show that $\pi_1$ is a cube modulo $\mathfrak{p}_2$ if and only if $f(\mathfrak{p}_1 , F^{(\mathfrak{p}_2)}/F)=1$. Our tools will be Artin reciprocity and the fundamental exact sequence of global class field theory, see \cite[Ch. V, Thm. 1.7]{milne}. We will use the following two facts \cite[p. 95, p. 105]{even2}:
\begin{enumerate}
\item The extension $F^{(\mathfrak{p}_2)}/F$ has conductor $3_1^2 \mathfrak{p}_2$.
\item The ray class group $H_{F}(3_1^2)$ is trivial.
\end{enumerate}
Applying the fundamental exact sequence to $F$ and the modulus $3_1^2 \mathfrak{p}_2$ gives a short exact sequence
\begin{equation} \label{ses}
1 \rightarrow \mathcal{O}_F^\times / (1+3_1^2 \mathfrak{p}_2)\cap \mathcal{O}_F^\times \rightarrow (\mathcal{O}_F / 3_1^2 \mathfrak{p}_2)^\times \rightarrow H(3_1^2 \mathfrak{p}_2) \rightarrow 1.
\end{equation}
Here we use that $F$ has class number $1$ to ensure that the last map is surjective. The Galois group of $F^{(\mathfrak{p}_2)}/F$ 
fits into the short exact sequence
\begin{equation} \label{ses2}
1 \rightarrow 3 \Gal(F(3_1^2 \mathfrak{p}_2)/F)  \rightarrow \Gal(F(3_1^2 \mathfrak{p}_2)/F) \rightarrow \Gal(F^{(\mathfrak{p}_2)}/F) \rightarrow 1
\end{equation}
where $3 \Gal(F(3_1^2 \mathfrak{p})/F)$ denotes the cubes in $\Gal(F(3_1^2 \mathfrak{p}_2)/F)$.  The link between the sequences (\ref{ses}) and (\ref{ses2}) is the Artin map $H(3_1^2 \mathfrak{p}_2) \xrightarrow{\sim} \Gal(F(3_1^2 \mathfrak{p}_2)/F)$, which sends the class of $\mathfrak{p}_1$ in $H(3_1^2 \mathfrak{p}_2)$ to the Frobenius element $\sigma_{\mathfrak{p}_1} \in \Gal(F(3_1^2 \mathfrak{p}_2)/F)$ of $\mathfrak{p}_1$. The condition $f(\mathfrak{p}_1 , F^{(\mathfrak{p}_2)}/F)=1$ is equivalent to $\sigma_{\mathfrak{p}_1}$ having trivial restriction to $\Gal(F^{(\mathfrak{p}_2)}/F)$. Using exactness of the sequences above, we see that $f(\mathfrak{p}_1 , F^{(\mathfrak{p}_2)}/F)=1$ if and only if there exists a unit $u \in \mathcal{O}_F^\times$ such that $u \pi_1$ is a cube in $(\mathcal{O}_F / 3_1^2 \mathfrak{p}_2)^\times$. By the Chinese remainder theorem, this is equivalent to $u \pi_1$ being a cube in $(\mathcal{O}_F / 3_1^2 )^\times$ and in $(\mathcal{O}_F / \mathfrak{p}_2)^\times$. By Lemma \ref{units_are_cubes_modulo_primes}, $u \pi_1$ is a cube in $(\mathcal{O}_F / \mathfrak{p}_2)^\times$ if and only if $\pi_1$ is a cube in $(\mathcal{O}_F / \mathfrak{p}_2)^\times$. By item (2) above and the fundamental exact sequence, the natural map $\mathcal{O}_F^\times \rightarrow (\mathcal{O}_F / 3_1^2 )^\times$ is surjective, so we can always find $u \in \mathcal{O}_F^\times$ such that $u \pi_1$ is a cube in $(\mathcal{O}_F / 3_1^2)^\times$. It follows that $f(\mathfrak{p}_1 , F^{(\mathfrak{p}_2)}/F)=1$ if and only if $\pi_1$ is a cube in $(\mathcal{O}_F / \mathfrak{p}_2)^\times$ as desired.
\end{proof}

\section{Short character sums}\label{short_character_sums}

Our argument relies on bounds for short character sums, and we now state a standard conjecture for these sums. If $\chi$ is a Dirichlet character modulo $q$, we define the incomplete character sum
\begin{equation*}
S_{\chi}(M,N) := \sum_{M< a \leq M+N} \chi(a)
\end{equation*}
for integers $M$ and $N$ with $N \geq 1$. When $\chi$ is non-principal, we should expect to find cancellation, and we make the following the conjecture (similar to \cite[Eqn. (9.4)]{FIMR}):

\begin{conjecture}[Conjecture $C_n$]
Let $n \geq 3$, $Q \geq 3$ and $N \leq Q^{\frac{1}{n}}$. For any non-principal cubic Dirichlet character $\chi$ of modulus $q \leq Q$ and $\varepsilon >0$, we have
\begin{equation*}
S_{\chi}(M,N) \ll_{\varepsilon,n} Q^{\frac{1-\delta}{n}+\varepsilon}
\end{equation*}
for all $M$ and some $\delta= \delta(n)>0$. The implied constant depends only on $\varepsilon$ and $n$.
\end{conjecture}

\noindent
Conjecture $C_3$ for cubefree moduli follows from Burgess bound \cite{burgess}. Conjecture $C_n$ is independent of GRH in the sense that it does not imply GRH, nor is it implied by GRH.\\

\noindent
We need a variant of Conjecture $C_n$ for arithmetic progressions, but only for some specific characters. Let $E$ be a number field containing $\zeta_3$ and let $\mathfrak{q}$ be a non-zero ideal of $\mathcal{O}_E$ such that $q:=N(\mathfrak{q})$ is a squarefree integer coprime to $3$. We then set $\chi_{\mathfrak{q}}(\ell):=\big( \frac{\ell}{\mathfrak{q}} \big)_{3,E}$ for any integer $\ell$. This is a Dirichlet character of modulus $q$. When $q>1$, it is non-principal.

\begin{lemma}\label{non_principal}
Suppose $q>1$. Then there is an integer $\ell$ coprime to $q$ such that $\chi_{\mathfrak{q}}(\ell)\neq 1$. 
\end{lemma}

\begin{proof}
Let $p$ be a prime factor of $q$. Since $N(\mathfrak{q})$ is squarefree, there must be a prime $\mathfrak{p}$ of $E$ which divides $\mathfrak{q}$, lies over $p$, and has $f_{\mathfrak{p}/p}=1$. Since $\zeta_3 \in E$, this forces $p$ to split in $\mathbb{Q}(\zeta_3)$, i.e. $p \equiv 1 \bmod{3}$ (note that $p\neq 3$ since $3\nmid q$). Hence the cubes have index $3$ in $\mathbb{F}_p^\times$ so we can choose $\ell_0$ not divisible by $p$ and not a cube modulo $p$. Since $q$ is squarefree, $p$ and $q/p$ are coprime so by the Chinese remainder theorem, we can choose $\ell$ such that $\ell \equiv \ell_0\bmod{p}$, and $\ell \equiv 1\bmod{q/p}$. With this choice of $\ell$, it is easy to see that $\chi_{\mathfrak{q}}(\ell)\neq 1$.
\end{proof}

Arguing as in \cite[Cor. 7]{Pell}, we deduce

\begin{corollary}\label{short_character_sum_corollary}
Let $\chi_{\mathfrak{q}}$ be as above, and assume Conjecture $C_n$. Then there exists $\delta=\delta(n)>0$ such that for all $\varepsilon>0$, the following holds: For all $Q \geq 3$, all $\mathfrak{q}$ as above with $N(\mathfrak{q}) \leq Q$ and all $N \leq Q^{\frac{1}{n}}$, we have
  \begin{equation*}
\sum_{\substack{
    M \leq a \leq M+N\\
    n \equiv l \bmod{k}
}} \chi_{\mathfrak{q}}(a) \ll_{\varepsilon ,n} Q^{\frac{1-\delta}{n}+\varepsilon}
\end{equation*}
for all integers $M,N,k,l$ with $N>0$ and $q \nmid k$. The implied constant depends only on $\varepsilon$ and $n$. 
\end{corollary}

\section{Vinogradov's sieve}\label{vinogradov_sieve}
We present the sieve, we will use to estimate $\sum_{N(\mathfrak{p}) \leq X} s_{\mathfrak{p}}$. Let $E$ be a number field and $(a_{\mathfrak{n}})_{\mathfrak{n}}$ a sequence of complex numbers labelled by the integral ideals of $E$. If $\mathfrak{m}$ is a non-zero integral ideal of $E$ and $X$ a positive real number, we define
\begin{equation*}
A_{\mathfrak{m}}(X):= \sum_{\substack{
    N(\mathfrak{n}) \leq X\\
    \mathfrak{m} \mid \mathfrak{n}
}} a_{\mathfrak{n}}.
\end{equation*}
When $M$ and $N$ positive real numbers, and $(v_{\mathfrak{m}})_{\mathfrak{m}}$ and $(w_{\mathfrak{n}})_{\mathfrak{n}}$ sequences of complex numbers satisfying $\lvert v_{\mathfrak{m}} \rvert, \lvert w_{\mathfrak{n}} \rvert \leq 1$, we define the bilinear sum
\begin{equation*}
B(M,N) := \sum_{N(\mathfrak{m}) \leq M} \sum_{N(\mathfrak{n}) \leq N} v_{\mathfrak{m}} w_{\mathfrak{n}} a_{\mathfrak{m} \mathfrak{n}}.
\end{equation*}
We refer to $A_{\mathfrak{m}}(X)$ as \emph{sums of type I} and to $B(M,N)$ as \emph{sums of type II}. The theorem below is sometimes known as Vinogradov's sieve (see for example \cite{peter}) because it originates from Vinogradov's work on representing odd integers as sums of three primes. The sieve in the form, we present, is \cite[Prop. 5.2]{FIMR}. 

\begin{theorem}\label{vinogradov}
Suppose $\lvert a_{\mathfrak{n}} \rvert \leq 1$ for all $\mathfrak{n}$, and we have fixed real numbers $0 < \vartheta , \theta <1$ such that for each $\varepsilon >0$ the following estimates hold:
\begin{equation*}
A_{\mathbf{m}}(X) \ll_\varepsilon X^{1-\vartheta +\varepsilon}
\end{equation*}
uniformly in all non-zero ideals $\mathfrak{m}$, and
\begin{equation*}
B(M,N) \ll_\varepsilon(M+N)^{\theta}(M N)^{1- \theta + \varepsilon}
\end{equation*}
uniformly in all sequences $(v_{\mathfrak{m}})_{\mathfrak{m}}$ and $(w_{\mathfrak{n}})_{\mathfrak{n}}$ satisfying $\lvert v_{\mathfrak{m}} \rvert, \lvert w_{\mathfrak{n}} \rvert \leq 1$. Then
\begin{equation*}
\sum_{N(\mathfrak{n}) \leq X} a_{\mathfrak{n}} \Lambda(\mathfrak{n}) \ll_\varepsilon X^{1-\frac{\vartheta \theta}{2+\theta}+\varepsilon}
\end{equation*}
for all $\varepsilon>0$. 
\end{theorem}

\noindent
Here $\Lambda$ is the natural generalization of the von Mangoldt function to number fields: $\Lambda(\mathfrak{n})= \log N(\mathfrak{p})$ if $\mathfrak{n} = \mathfrak{p}^r$ for some prime ideal $\mathfrak{p}$ and integer $r \geq 1$, otherwise $\Lambda(\mathfrak{n})=0$. Using partial summation, one deduces an estimate of the form $\sum_{N(\mathfrak{p}) \leq X} a_{\mathfrak{p}} \ll X^{1- \delta}$ for some positive $\delta$. In spin problems, estimates for sums of type I are often conditional on conjectures on short character sums, whereas estimates for sums of type II are unconditional.\\

\noindent
For our application, we take $E = F'$, and $a_{\mathfrak{n}}$ will essentially be the spin symbol $s_{\mathfrak{n}}$. However, we would like to prove that the values of $s_{\mathfrak{n}}$ equidistribute in any abelian Chebotarev class, and we therefore do the following: Let $h_0:=\# H_0$, and fix distinct prime ideals $\mathfrak{p}_1$, ..., $\mathfrak{p}_{h_0}$ of degree $1$ over $\mathbb{Q}$  and coprime to $3$ that represent the classes of $H_0$. If the class of an ideal $\mathfrak{n}$ lies in $H_0$, we have $\mathfrak{n}\mathfrak{p}_i=(\alpha)$ for some $i \in \{1,\ldots,h_0\}$ and some $\alpha\in\mathcal{O}_{F'}$. For $i \in \{1,....,h_0\}$, a non-zero ideal $\mathfrak{M}$ of $\mathcal{O}_{F'}$, and $\mu\in(\mathcal{O}_{F'}/\mathfrak{M}\mathcal{O}_{F'})^{\times}$, we set
\begin{equation*}
r_i(\mathfrak{n},\mathfrak{M},\mu) := 
\begin{cases}
1 & \textrm{if }\mathfrak{n}\mathfrak{p}_i=(\alpha)\textrm{ for some }\alpha\equiv\mu\pmod{\mathfrak{M}}\\
0 & \textrm{otherwise}
\end{cases}.
\end{equation*}
For fixed $i$, $\mathfrak{M}$ and $\mu$, we take $a_{\mathfrak{n}}:=r_i(\mathfrak{n},\mathfrak{M},\mu)s_{\mathfrak{n}}$ and prove the following estimates:
\begin{proposition}\label{sums_of_type_I}
Assume Conjecture $C_{12}$. Then there is $\vartheta>0$ such that for all $\varepsilon>0$, we have
\begin{equation*}
\sum_{\substack{N(\mathfrak{n})\leq X\\ \mathfrak{m}\mid\mathfrak{n}}}r_i(\mathfrak{n},\mathfrak{M},\mu)s_{\mathfrak{n}}\ll_{\varepsilon} X^{1-\vartheta+\varepsilon}
\end{equation*}
uniformly in all non-zero ideals $\mathfrak{m}$ of $\mathcal{O}_{F'}$.
\end{proposition}

\begin{proposition}\label{sums_of_type_II}
We have 
\begin{equation*}
\sum_{N(\mathfrak{m})\leq M}\sum_{N(\mathfrak{n})\leq N} v_{\mathfrak{m}}w_{\mathfrak{n}} r_i(\mathfrak{mn},\mathfrak{M},\mu)s_{\mathfrak{m}\mathfrak{n}} \ll_{\varepsilon} (M+N)^{\frac{1}{48}}(MN)^{1-\frac{1}{48}+\varepsilon}
\end{equation*}
uniformly in all sequences $(v_{\mathfrak{m}})_{\mathfrak{m}}$ and $(w_{\mathfrak{n}})_{\mathfrak{n}}$ of complex numbers with modulus at most $1$.
\end{proposition}

\noindent
By now, there are many general results in the literature that can be used to estimate sums of type II, and, in our case, Proposition \ref{sums_of_type_II} is a consequence of \cite[Proposition 4.3]{Koymans_Rome_2024}. The hardest part of our argument is to prove Proposition \ref{sums_of_type_I}, and we must improve on existing techniques to succeed. Given Proposition \ref{sums_of_type_I} and Proposition \ref{sums_of_type_II}, we deduce Theorem \ref{main_theorem} from Theorem \ref{vinogradov} and partial summation.

\section{A fundamental domain}
Because ray class groups are finite, we will eventually reduce the problem of estimating the sums $A_{\mathfrak{m}}(X)$ and $B(M,N)$ to estimating sums over principal ideals of $F'$ with generators satisfying certain congruence conditions. Generators of principal ideals are only unique up to multiplication by units. The unit group $\mathcal{O}_{F'}^\times$ decomposes as $T \times V$ where $T$ is the torsion subgroup of $\mathcal{O}_{F'}^\times$ and $V$ is a free abelian group. In fact $T = \left\langle \xi_6 \right\rangle$, and $V$ has rank $3$. We fix one such decomposition $\mathcal{O}_{F'}^\times = T \times V$. The purpose of this section is to give a fundamental domain for the action of $V$ on $\mathcal{O}_{F'}$ consisting of elements that are not too large.\\

\noindent
Let $\eta = \{\eta_1, ..., \eta_8\}$ be an integral basis for $\mathcal{O}_{F'}$. We can embed $F' \hookrightarrow \mathbb{R}^8$ by sending $a_1 \eta_1+ \cdots + a_8 \eta_8$ to $(a_1 ,..., a_8)$. To measure sizes, we define a polynomial $f$ in variables $X_1,...,X_8$ by $f(X_1,...,X_8):= N_{F'/\mathbb{Q}}(X_1 \eta_1 +\cdots + X_8 \eta_8)$. When $S \subset \mathbb{R}^8$ and $X > 0$, we set $S(X):= \left\{(x_1 , ..., x_8) \in S\,:\, \lvert f(x_1 , ..., x_8) \rvert \leq X \right\}$. By \cite[Lemma 3.5]{imrn} we have: 

\begin{lemma}\label{fundamental_domain}
There exists a subset $\mathcal{D} \subset \mathbb{R}^8$ with the following properties:
\begin{enumerate}
\item For all $\alpha \in \mathcal{O}_{F'} \setminus \left\{ 0 \right\}$, there exists a unique $v \in V$ such that $v \alpha \in \mathcal{D}$. Moreover, if $u\in\mathcal{O}_{F'}^{\times}$, we have $u \alpha \in \mathcal{D}$ if and only if $u\in vT$.
\item $\mathcal{D}(1)$ has $7$-Lipschitz parametrizable boundary.
\item There exists a constant $C_{\eta}>0$ such that when $\alpha = a_1 \eta_1 +\cdots a_8 \eta_8\in\mathcal{D}$ with $a_1,..., a_8 \in \mathbb{Z}$, we have $\lvert a_i \rvert \leq C_{\eta} \lvert N_{F'/\mathbb{Q}}(\alpha) \rvert^{\frac{1}{8}}$ for all $i=1,...,8$. 
\end{enumerate}
\end{lemma}

\noindent
In particular, each non-zero principal ideal has exactly $|T|=6$ generators in $\mathcal{D}$.

\section{Counting ideals of squarefull norm}\label{counting_ideals}
Let $M$ be a number field of degree $2n$ over $\mathbb{Q}$, and suppose $\Lambda \subset \mathcal{O}_M$ is a lattice in $M$ of rank $n$ such that $\alpha \Lambda$ is not contained in a proper subfield of $M$ for any $\alpha\in M^{\times}$ (in particular, we must have $n\geq 2$). If $a$ is a positive integer, we can uniquely write $a=q g$ where $q$ is squarefree, $g$ is squarefull and $\gcd(q,g)=1$. We call $g$ the \emph{squarefull part of $a$} and denote it $\mathrm{sqfull}(a)$. The purpose of this section is to estimate the number of elements in $\Lambda$ of bounded size and whose norm onto $\mathbb{Q}$ has large squarefull part. We encounter this problem when estimating sums of type I, and when the rank of the lattice $\Lambda$ is exactly half of the degree of $M$, existing results such as \cite[Lemma 3.1]{Koymans} fall short of giving a non-trivial estimate. This section can be read independently of the rest of the paper. Moreover, the notation is specific to this section and can coincide with the notation used in other places. \\

\noindent
It is necessary to impose that $\alpha \Lambda$ is not contained in a proper subfield of $M$ for any $\alpha \in M^\times$, because otherwise $N_{M /\mathbb{Q} }(\alpha \lambda)$ is squarefull for all $\lambda \in \Lambda $ so the squarefull part $N_{M/ \mathbb{Q}}(\lambda)$ is always large. If we fix a $\mathbb{Z}$-basis $\omega_1 ,..., \omega_n$ for $\Lambda$, this condition is equivalent to $\frac{1}{\omega_1} \Lambda$ not being contained in a proper subfield of $M$. Indeed, let $N$ be the Galois closure of $M/\mathbb{Q}$. If $\alpha \Lambda$ is contained in a proper subfield, then there is $\sigma \in \mathrm{Gal}(N/\mathbb{Q})- \mathrm{Gal}(N/M)$ that fixes $\alpha \omega_1,..., \alpha \omega_n$. In particular, $\sigma (\omega_i / \omega_1)= \sigma( \alpha \omega_i ) / \sigma(\alpha \omega_1)=\omega_i / \omega _1$ for all $i=2,...,n$ so it follows that $\frac{1}{\omega_1} \Lambda$ is contained in the same subfield.\\  

\noindent
Every $\lambda \in \Lambda$ can be written uniquely as $a_1 \omega_1 + \cdots +a_n \omega_n$ where $a_1, ..., a_n \in \mathbb{Z}$, and given positive real numbers $L,Z>0$, the task is to estimate the size of the set
\begin{equation*}
\left\{ \lambda \in \Lambda \, :\, \lvert a_i \rvert \leq L,\, \mathrm{sqfull}(N_{M/\mathbb{Q}}(\lambda)) \geq  Z \right\}.
\end{equation*}
The goal is to improve over the trivial bound by a power saving, and the main result is the following proposition:

\begin{proposition}\label{squarefull}
Let $L$ and $Z$ be positive real numbers. Then there is $\theta >0$ depending only on the chosen $\mathbb{Z}$-basis for $\Lambda$ such that for all $\varepsilon >0$, we have
\begin{equation*}
\# \left\{ \lambda \in \Lambda \,:\, \lvert a_i \rvert \leq L,\, \mathrm{sqfull}(N_{M/\mathbb{Q}}(\lambda)) \geq Z \right\} \ll_\varepsilon L^{n+ \varepsilon } Z^{- \theta}.
\end{equation*}
\end{proposition}

\noindent
If the rank of $\Lambda$ had been strictly greater than $n$, say $\Lambda = \mathbb{Z}\omega_1\oplus\cdots\oplus\mathbb{Z}\omega_k$ where $k>n$, then the proof of \cite[Lemma 3.1]{Koymans} shows that 
\begin{equation}\label{normal_lattice_counting}
\# \{\lambda \in \Lambda\,:\,|a_i|\leq L,\,\mathrm{sqfull}(N_{M/\mathbb{Q}}(\lambda))\geq Z\}\ll_{\varepsilon} L^{k+\varepsilon}Z^{1-\frac{k}{n}}
\end{equation}
for all $\varepsilon>0$. We will also need this bound when estimating sums of type I.\\

\noindent
Proving Proposition \ref{squarefull} requires some preliminary results. The central element of the proof is the \emph{norm polynomial} (or the \emph{norm form}):
\begin{equation*}
F (X_1,...,X_n) := N_{M/\mathbb{Q}}(X_1 \omega_1 + \cdots + X_n \omega_n) \in\mathbb{Z}[X_1,...,X_n].
\end{equation*}
Our condition that $\frac{1}{\omega_1}\Lambda$ is not contained in a proper subfield of $M$ exactly means that $F$ is irreducible over $\mathbb{Q}$ c.f. \cite[Ch. VII Lemma 1B]{schmidt_book}.\\

\noindent
We will eventually have to count the number of solutions $(c_1,...,c_n)$ to the equation $F(c_1,...,c_n)=kg$ where $k$ is small, $g\leq L^{2n}$ is squarefull, and $|c_i|\leq L$ for all $i=1,...,n$. Any squarefull number can be written uniquely as $z^3y^2$ where $z$ is squarefree. When $z$ is small, we use an effective version of the Hilbert irreducibility theorem to count the number of solutions for each fixed $z$. When $z$ is large, the number of possibilities for $y$ is limited, and, for each $y$ and $z$, the solutions to the norm equation $F(c_1,...,c_n)=kz^3y^2$ are very sparse so a simple counting argument gives the desired power saving.\\

Our application of the Hilbert irreducibility theorem is the following lemma:

\begin{lemma}\label{effective_hilbert}
For a non-zero integer $a$ and real number $L>0$, let $M_a(L)$ denote the number of tuples $(x_1 ,..., x_n) \in \left[ -L,L \right]^n \cap \mathbb{Z}^n$ such that $F(x_1 ,..., x_n)=a y^2$ for some integer $y$. Then there are constants $C(F,n,a), D(F,n,a)>0$ depending at most polynomially on the coefficients of $F$, $n$ and $a$ such that $M_a(L ) \leq C(F,n,a) L^{n -\frac{1}{2}} \log L$ for all $L \geq D(F,n,a)$.
\end{lemma}

\begin{proof}
Fix $a$, and let $G(\mathbf{X} ,Y) := a Y^2 -F(\mathbf{X} ) \in \mathbb{Z} [ \mathbf{X},Y]$ where $\mathbf{X}  =(X_1 ,..., X_n)$. Since $F$ is irreducible over $\mathbb{Q}$, it follows that $G$ is irreducible over $\mathbb{Q} $. If $\mathbf{x} \in \mathbb{Z}^n$, and $F(\mathbf{x})= a y^2$ for some $y \in \mathbb{Z}$, the polynomial $G(\mathbf{x} ,Y)$ is reducible over $\mathbb{Q}$. The result now follows by an effective version of the Hilbert irreducibility theorem due to Cohen \cite[Theorem 2.5]{cohen}.
\end{proof}

\noindent
We now consider the norm form equation $F(x_1,...,x_n)=a$ where $a\in\mathbb{Z}$.
For non-degenerate lattices $\Lambda$, 
the results in \cite[Ch. VII]{schmidt_book} on the number of solutions to norm form equations 
have effective versions \cite{schmidt_effective}.
In our case, the lattice $\Lambda$ is allowed to be degenerate and the norm equation has infinitely many solutions, making the effective results (ibid.) unavailable. Instead, we use Schmidt's subspace theorem \cite[Theorem 2]{schmidt} to count the number of solutions when they are restricted to lie in a box in $\mathbb{R}^n$. We start with a general lemma.

\begin{lemma}\label{units_are_sparse}
Let $K$ be a number field of degree $m$ over $\mathbb{Q}$ with integral basis $\eta = \left\{ \eta_1,...,\eta_m \right\}$. Write $m = r +2s$ where $r$ and $2s$ are the number of real and complex of embeddings of $K$ respectively. For $L \geq 1$, let $B_{L} $ be a the set of non-zero elements in $\mathcal{O}_K$ of the form $a_1 \eta_1 + \cdots + a_m \eta_m$ with $a_1,..., a_m \in \mathbb{Z}$, and $\lvert a_i \rvert  \leq L$ for all $i=1,...,m$. Then there is a constant $C_{K,\eta}$ depending only on $K$ and $\eta$ such that
\begin{equation*}
\# \left\{ u \in \mathcal{O}_K^{\times}\,:\, u B_{L}\cap B_{L} \neq \emptyset \right\} \leq C_{K,\eta}(\log L)^{r+s-1}.
\end{equation*}  
\end{lemma}

\begin{proof}
Let $x \mapsto x^{(i)}$ denote the real embeddings of $K$ for $i=1,...,r$ and $s$ pairwise non-conjugate complex embeddings for $i=r+1,...,r+s$. As the statement of the lemma very strongly suggests, we should consider the  logarithmic map $\mathcal{L}: K^\times \rightarrow \mathbb{R}^{r+s}$ defined by
\begin{equation*}
\mathcal{L}(x) :=(\log \lvert x^{(1)} \rvert,..., \log \lvert x^{(r)} \rvert, 2 \log \lvert x^{(r+1)} \rvert,..., 2 \log \lvert x^{(r+s)} \rvert).
\end{equation*}
This is a homomorphism whose kernel is the roots of unity in $K$, and we must estimate the number of units $u \in \mathcal{O}_{K}^\times$ such that $\mathcal{L}(B_{L})\cap (\mathcal{L}(u)+\mathcal{L}(B_{L}))\neq \emptyset$.

The main claim is that there is a constant $C$ (depending only on $K$ and $\eta$) such that $\mathcal{L}(B_{L })\subset \left[ -C \log L, C \log L \right]^{r+s}$. To see why this is sufficient, recall that $\mathcal{L}(\mathcal{O}_{K }^\times )$ is a lattice of rank $r+s-1$. Hence $\mathcal{L}(\mathcal{O}_{K}^\times)$ contains at most $D(\log L)^{r+s-1}$ vectors of sup-norm at most $2 C \log L$ where $D$ is a constant only depending on $K$ and $\eta$ (see for example \cite[Theorem 5.4]{lattice_counting}). If $v \in \mathcal{L}(\mathcal{O}_{K}^\times)$ has sup-norm greater than $2 C \log L$ then clearly, $\mathcal{L}(B_{L})+v$ is disjoint from $\mathcal{L}(B_{L})$. Therefore, we can take $C_{K,\eta}=T D$ where $T$ is number of roots of unity in $K$.

To find $C$ as above, we use Schmidt's subspace theorem in the form of \cite[Theorem 2]{schmidt} to find a constant $c>0$ (depending only on $K$ and $\eta$ such that for all $i=1,...,r+s$, we have
\begin{equation*}
|a_1 \eta_1^{(i)}+\cdots+a_m \eta_m^{(i)}| \geq c L^{-m-1}    
\end{equation*}
for all integers $a_1,...,a_m$ with $0<\max\{|a_1|,...,|a_m|\}\leq L$. We remark that \cite[Theorem 2]{schmidt} is stated in terms of real algebraic numbers, but, from this, one can deduce a version for complex algebraic numbers. We clearly have the upper bound $\ll L$ for the same expression so it is now clear that $\mathcal{L}(B_L)\subset [-C\log L,C\log L]^{r+s}$ for some $C$ only depending on $K$ and $\eta$.
\end{proof} 

We deduce the following lemma:

\begin{lemma}\label{schmidt}
For a non-zero integer $a$ and real number $L>0$, let $N_a(L)$ denote the number of tuples $(x_1 ,..., x_n) \in \left[ -L,L \right]^n \cap \mathbb{Z}^n$ such that $F(x_1 ,..., x_n)=a$. Then for all $\varepsilon>0$, there is a constant $C(\varepsilon , F)$ depending only on $\varepsilon$ and $F$ such that $N_a(L) \leq C(\varepsilon,F) \lvert a \rvert^\varepsilon L^\varepsilon$ for all $a$ and $L$.      
\end{lemma}

\begin{proof}
Fix a non-zero integer $a$. By definition of $F$, we must estimate the number of $\lambda \in \Lambda $ with bounded coefficients such that $N_{M/ \mathbb{Q} }(\lambda)=a$, but it turns out to be enough to count $\alpha \in \mathcal{O}_M$ with bounded coefficients and $N_{M/\mathbb{Q} }(\alpha)=a$. If $N_{M/ \mathbb{Q} }(\alpha)=a$, the ideal generated by $\alpha$ has norm $\lvert a \rvert$. For any $\varepsilon>0$, there are at most $C_\varepsilon \lvert a \rvert^{\varepsilon}$ such ideals. Fix a principal ideal $(\alpha)$ of norm $\lvert a \rvert$. Extend the $\mathbb{Z}$-basis $\omega_1 ,..., \omega_n$ of $\Lambda$ to a $\mathbb{Q}$-basis of $M$ by adding the elements $\omega_{n+1},..., \omega_{2n} \in \mathcal{O}_M$, and choose an integral basis $\eta_1 ,..., \eta_{2n}$ for $\mathcal{O}_M$. Then there is a constant $A>0$ only depending on $\omega_1,...,\omega_{2n}$ and $\eta_1,...,\eta_{2n}$ such that
  \begin{equation*}
\begin{split}
&\left\{ a_1 \omega_{1} +\cdots + a_{2n} \omega_{2n} \in \mathcal{O}_{M}\,|\, a_i \in \mathbb{Q},\, \lvert a_i \rvert \leq L \right\}\\ \subset & \left\{ b_1 \eta_1 + \cdots b_{2n} \eta_{2n} \in \mathcal{O}_{M}\,|\, b_i \in \mathbb{Z},\, \lvert b_i \rvert \leq A L \right\}.
\end{split}
\end{equation*}
Hence if $\alpha = a_1 \eta_1 + \cdots + a_{2n} \eta_{2n}$, with $a_i \in \mathbb{Z}$ and $\lvert a_i \rvert \leq L$, it follows by Lemma \ref{units_are_sparse} that there are at most $C_{\varepsilon,F} L^{\varepsilon}$ units $u \in \mathcal{O}_{K}^\times $ such that the coefficients of $u \alpha$ with respect to $\omega_1 ,..., \omega_{2n}$ are all smaller than $L$ in absolute value. Hence there are at most $C_{\varepsilon} C_{\varepsilon,F} \lvert a \rvert^{\varepsilon} L^{\varepsilon}$
elements $\alpha \in \mathcal{O}_M$ with norm $a$ and coefficients with respect to $\omega_1 ,..., \omega_{2n}$ bounded by $L$. 
\end{proof}

Combining the previous lemmas, we prove the following technical result which is the key new input. 

\begin{lemma}\label{squarefull_values}
For positive real numbers $L$ and $\delta$, let $N_{\delta}(L)$ denote the number of tuples $(c_1 , ..., c_n)$ in $\mathbb{Z}^n $ such that $\lvert c_j \rvert \leq L^{1+ \delta}$ for all $j=1,...,n$ and $F(c_1 , ..., c_n) = kg$  for some $1 \leq k \leq L^{2n\delta}$ and some squarefull integer $g$. Then there exists $\theta >0$ depending only on $\Lambda$ such that for $\delta$ small enough (only in terms of $\Lambda$), we have $\lvert N_{\delta}(L)\rvert \ll L^{n-\theta}$.
\end{lemma}

\begin{proof}
For a fixed $k \leq L^{2n\delta}$, let $N_{\delta,k}(L)$ denote the set of tuples $(c_1 ,..., c_n) \in \mathbb{Z}^n$ with $\lvert c_j \rvert \leq L^{1+\delta}$, and $F(c_1 ,..., c_n)=k g$ for some squarefull $g$. By making $\delta$ small enough, it is enough to prove $\lvert N_{\delta,k}(L)\rvert \ll L^{n- \theta}$ for each $k$ where $\theta >0$  depends only on $\Lambda$ (and not on $k$). Any squarefull number can be written uniquely as $z^3 y^2$ for positive integers $y$ and $z$ with $z$ squarefree. We partition $N_{\delta,k }(L)$ as $\sqcup_z N_{\delta,k,z}(L)$ where $N_{\delta,k,z}(L)$ is defined in the same way as $N_{\delta,k}(L)$ but with $F(c_1 , ..., c_n)=kz^3 y^2$ for some $y$. Suppose $z \leq L^\eta$ for some small $\eta >0$ that is to be determined. By Lemma \ref{effective_hilbert}, there are constants $C,D>0$ such that $\lvert N_{\delta,k,z}(L)\rvert \leq C(k z^3)^D L^{(1+ \delta)(n-\frac{1}{2})}\log(L^{1+\delta})$ for $L$ large enough in terms of $\Lambda$. Since $k \leq L^{2n \delta}$, we have
\begin{equation*}
\lvert N_{\delta,k,z}(L) \rvert \ll_{\varepsilon , \Lambda} z^{3D} L^{n+(2nD+n-\frac{1}{2})\delta+ (1+ \delta) \varepsilon -\frac{1}{2}}
\end{equation*}
for all $\varepsilon >0$. Hence
\begin{equation*}
\sum_{1 \leq z \leq L^\eta} \lvert N_{\delta,k,z}(L) \rvert \ll_{\varepsilon, \Lambda} L^{n+(3 D+1)\eta +(2n D +n-\frac{1}{2})\delta+(1+ \delta) \varepsilon-\frac{1}{2}}
\end{equation*}
for all $\varepsilon >0$. When $z > L^\eta$, we estimate $\lvert N_{\delta,k,z}(L)\rvert$ in a different way. We must bound the number of tuples $(c_1 , ..., c_n) \in \mathbb{Z}^n$ with $\lvert c_j \rvert \leq  L^{1+\delta}$ such that $F(c_1 , ..., c_n)=k z^3 y^2$ for some $y \in \mathbb{Z}$. For such a tuple $(c_1 ,..., c_n)$, we have $|F(c_1 , ..., c_n )| \leq C_\Lambda L^{2n(1+ \delta)}$ where $C_{\Lambda }>0$ only depends on $\Lambda$. Therefore,
\begin{equation*}
y \ll_{\Lambda } \left( \frac{L^{2n(1+\delta)}}{kz^3} \right)^{\frac{1}{2}} \leq \frac{L^{n(1+ \delta)}}{z^{\frac{3}{2}}} < \frac{L^{n+n \delta-\eta/4}}{z^{\frac{5}{4}}},
\end{equation*}
so the number of possibilities for $y$ is at most this number with implied constants depending only on $\Lambda$. For any integer $a \leq C_\Lambda L^{2n}$ there are $\ll_{\varepsilon, \Lambda} L^{(1+ \delta)\varepsilon}$ tuples $(c_1 , ..., c_n ) \in \mathbb{Z}^n $ such that $F(c_1 , ..., c_n)=a$ by Lemma \ref{schmidt}. It follows that
\begin{equation*}
\sum_{z>L^\eta} \lvert N_{\delta , k, z}(L) \rvert \ll_{\varepsilon, \Lambda} L^{n+n \delta+(1+\delta)\varepsilon -\eta/4} \sum_{z>L^\eta} z^{-\frac{5}{4}} \ll_{\varepsilon , \Lambda } L^{n+ n \delta+(1+\delta)\varepsilon -\eta /4}
\end{equation*}
for all $\varepsilon >0$ since the sum over $z$ converges. Hence
\begin{equation*}
\lvert N_{\delta,k,z}(k,z)\rvert \ll_{\varepsilon, \Lambda} L^{n +(3D+1)\eta +(2n D+ n -\frac{1}{2})\delta +(1+ \delta)\varepsilon-\frac{1}{2}}+L^{n+n \delta +(1 + \delta)\varepsilon -\eta /4}, 
\end{equation*}
and since we can choose $\eta$ independently of $\delta$, we obtain a power-saving when $\delta$ is small enough in terms of $\Lambda$. \end{proof}

We can now prove the main result of this section.

\begin{proof}[Proof of Proposition \ref{squarefull}]
Let $\iota: \Lambda  \hookrightarrow \mathbb{R}^n $ be the embedding $\iota(a_1 \omega_1 +\cdots + a_n \omega_n)=(a_1 , ..., a_n)$ so that we identify $\Lambda $ with $\mathbb{Z}^n$ inside $\mathbb{R}^n$. Let 
\begin{equation*}
S_L:= \left\{(x_1 , ..., x_n) \in \mathbb{R}^n \,:\, \lvert x_i \rvert \leq  L \right\}\subset \mathbb{R}^n
\end{equation*}
be the standard hypercube scaled by $L$. For $\mathfrak{g}$ a non-zero ideal of $\mathcal{O}_M$, we set $\Lambda_{\mathfrak{g}}:= \Lambda \cap \mathfrak{g}$ . The number of interest is therefore bounded by
\begin{equation*}
\sum_{\substack{ Z \leq N_{M/ \mathbb{Q}}(\mathfrak{g}) \leq L^{2n}\\ N_{M/\mathbb{Q} }(\mathfrak{g}) \textrm{ squarefull} 
}} \lvert S_L\cap \iota(\Lambda_{\mathfrak{g}}) \rvert.
\end{equation*}
For a lattice $\Lambda_0 \subset \mathbb{R}^n$, $\lambda_1(\Lambda_0),...,\lambda_n(\Lambda_0 )$ denote the successive minima of $\Lambda_0$, i.e. $\lambda_i(\Lambda_0)$ is the smallest real number $l$ such that $\Lambda_0$ contains $i$ linearly independent vectors of Euclidean length at most $l$. For more on this see \cite[Section 4]{lattice_counting}. By \cite[p. 1741]{Koymans}, there is a constant $C>0$ only depending on the number field $M$ such that $\lambda_1(\Lambda_{\mathfrak{g}}) \geq C N_{M/\mathbb{Q}}(\mathfrak{g})^{\frac{1}{2n}}$ for all $\mathfrak{g}$. Unlike in \cite{Koymans}, this bound is not always strong enough for our purpose. Instead, we show that, often enough, $\lambda_1(\Lambda_{\mathfrak{g}})$ is larger than $N_{M /\mathbb{Q} }(\mathfrak{g})^{(1+\delta)/2n}$ for some small $\delta>0$ that is to be determined. We split the above sum into two parts
\begin{equation*}
\sum_{\substack{
    Z \leq N_{M/\mathbb{Q} }(\mathfrak{g}) \leq L^{2n}\\
    N_{M/\mathbb{Q} }(\mathfrak{g}) \textrm{ squarefull}\\
    \lambda_1(\Lambda_{\mathfrak{g}}) \geq N_{M/\mathbb{Q} }(\mathfrak{g})^{\frac{1+ \delta}{2n}}
}} \lvert S_L \cap \iota(\Lambda_{\mathfrak{g}}) \rvert  + \sum_{\substack{
Z \leq N_{M/\mathbb{Q}}(\mathfrak{g}) \leq L^{2n}\\ N_{M/ \mathbb{Q}}(\mathfrak{g}) \textrm{ squarefull}\\ \lambda_1(\Lambda_{\mathfrak{g}}) < N_{M/\mathbb{Q}}(\mathfrak{g})^{\frac{1+\delta}{2n}} } 
} \lvert S_L \cap \iota(\Lambda_{\mathfrak{g}}) \rvert
\end{equation*}
which we estimate separately. Estimating the first sum will be very similar to the proof of \cite[Lemma 3.1]{Koymans}, but estimating the second sum requires new ideas. We start by briefly explaining how the first sum is estimated. Fix $\mathfrak{g}$ satisfying the conditions in the first sum. By \cite[Theorem 5.4]{lattice_counting} and Minkowski's second theorem \cite[Theorem V, p. 218]{geometry_of_numbers}, the same argument as \cite[p. 1741]{Koymans} gives
\begin{equation*}
\lvert S_L \cap \iota(\Lambda_{\mathfrak{g}}) \rvert \ll \frac{L^n}{N_{M/ \mathbb{Q} }(\mathfrak{g})^{\frac{1+\delta}{2}}}.
\end{equation*}
where the implied constant only depends on $M$. Estimating as in \cite[p. 1742]{Koymans}, we find that the first sum is $ \ll_\varepsilon L^{n+\varepsilon} Z^{-\frac{\delta}{2}}$ for all $\varepsilon >0$ so we have obtained a non-trivial saving.\\

\noindent
To estimate the second sum, we perform a dyadic decomposition:
\begin{equation*}
\sum_{\substack{
    Z \leq N_{M/\mathbb{Q}}(\mathfrak{g}) \leq L^{2n}\\
    N_{M /\mathbb{Q}}(\mathfrak{g})\textrm{ squarefull}\\
    \lambda_1(\Lambda_{\mathfrak{g}}) < N_{M/ \mathbb{Q}}(\mathfrak{g})^{\frac{1+\delta}{2n}}
}} \lvert S_{L} \cap \iota(\Lambda_{\mathfrak{g}}) \rvert = \sum_{\substack{
i \geq 0\\ Z \leq 2^i \leq L^{2n} 
}} \quad \sum_{\substack{
2^i \leq N_{M/\mathbb{Q}}(\mathfrak{g}) < 2^{i+1}\\ N_{M/\mathbb{Q}}(\mathfrak{g}) \textrm{ squarefull}\\ \lambda_1(\Lambda_{\mathfrak{g}})< N_{M/\mathbb{Q}}(\mathfrak{g})^{\frac{1+\delta}{2n}} 
}} \lvert S_L \cap \iota(\Lambda_{\mathfrak{g}}) \rvert.
\end{equation*}
By \cite[p. 1741]{Koymans} we have $\lambda_1(\Lambda_{\mathfrak{g}}) \gg N_{M/ \mathbb{Q}}(\mathfrak{g})^{\frac{1}{2n}}$ for all $\mathfrak{g}$, and the arguments in \cite{Koymans} that comes after this observation show that $\lvert S_L \cap \iota(\Lambda_{\mathfrak{g}}) \rvert \ll L^n / N_{M/\mathbb{Q}}(\mathfrak{g})^{1/2}$ with the implied constant only depending on $M$. Hence the above is bounded by
\begin{equation}\label{dyadic_decomposition}
L^n \sum_{\substack{
i \geq 0\\ Z \leq 2^i \leq L^{2n}}} 2^{-i/2} \sum_{\substack{
2^{i} \leq N_{M/ \mathbb{Q}}(\mathfrak{g})< 2^{i+1}\\ N_{M /\mathbb{Q}}(\mathfrak{g}) \textrm{ squarefull}\\ \lambda_1(\Lambda_{\mathfrak{g}})<N_{M/\mathbb{Q}}(\mathfrak{g})^{\frac{1+ \delta}{2n}} 
}} 1.
\end{equation}
We now claim that for each $i$, we have
\begin{equation*}
 \sum_{\substack{
2^{i} \leq N_{M/ \mathbb{Q}}(\mathfrak{g})< 2^{i+1}\\ N_{M /\mathbb{Q}}(\mathfrak{g}) \textrm{ squarefull}\\ \lambda_1(\Lambda_{\mathfrak{g}})<N_{M/\mathbb{Q}}(\mathfrak{g})^{\frac{1+ \delta}{2n}} 
}} 1 \ll_{\varepsilon} L^{2n \delta + \varepsilon} \lvert  N_{\delta}(2^{\frac{i+1}{2n}}) \rvert
\end{equation*}
for all $\varepsilon  >0$ where the set $N_{\delta}(2 ^{\frac{i+1}{2n}})$ is defined in Lemma \ref{squarefull_values}. Fix $i$, and suppose that $\mathfrak{g}$ is an ideal satisfying the conditions in the inner sum above. Let $\lambda_{\mathfrak{g}} = c_1 \omega_1 +\cdots + c_n \omega_n$ be any non-zero vector in $\Lambda_{\mathfrak{g}}$ of length $\lambda_1(\Lambda_{\mathfrak{g}})$. Clearly, $N_{M/ \mathbb{Q}}( \mathfrak{g})$ divides $N_{M/ \mathbb{Q}}(\lambda_{\mathfrak{g}})$ and, moreover, the quotient is at most $L^{(i+1) \delta}$: Indeed, let $F(X_1 ,..., X_n)$ be the norm polynomial so that $N_{M / \mathbb{Q}}(\lambda_{\mathfrak{g}})= F(c_1 ,..., c_n)$. Then
\begin{equation*}
F(c_1 , ..., c_n) \ll \max_{1 \leq j \leq n} \lvert c_j \rvert^{2n} \ll \lambda_1(\Lambda_{\mathfrak{g}})^{2n} < N_{M/ \mathbb{Q} }(\mathfrak{g})^{1+ \delta} \leq N_{M / \mathbb{Q} }(\mathfrak{g}) L^{(1+i)\delta }.
\end{equation*}
Here we use that the size of each $c_j$ is at most the Euclidean norm of $\iota(\lambda_{\mathfrak{g}})$. Since $\lambda_1(\Lambda_{\mathfrak{g}})< N_{M/\mathbb{Q}}(\mathfrak{g})^{\frac{1+\delta}{2n}}$, and $N_{M /\mathbb{Q}}(\mathfrak{g}) < 2^{i+1}$, it follows that $\lvert c_j \rvert< 2^{\frac{i+1}{2n}}$. Hence $(c_1 , ..., c_n) \in N_{\delta}(2^{\frac{i+1}{2n}})$, so we must show that the fibers of the assignment $\mathfrak{g} \mapsto(c_1 , ..., c_n)$ have size $\ll_{\varepsilon} L^{2n \delta + \varepsilon}$ for all $\varepsilon >0$. Since the quotient of $F(c_1 , ..., c_n)$ and $N_{M/\mathbb{Q} }(\mathfrak{g})$ is at most $L^{(i+1) \delta} \ll L^{2n \delta}$ by the above, there are at most $L^{2n \delta}$ possibilities for $N_{M/\mathbb{Q}}(\mathfrak{g})$. Given $g \in \mathbb{N}$ there are at most $\ll_\varepsilon g^\varepsilon$  ideals of $\mathcal{O}_M$ of norm $g$ for any $\varepsilon >0$ so the number of possibilities for $\mathfrak{g}$ is at most $L^{2n \delta + \varepsilon}$ as desired. By Lemma \ref{squarefull_values}, there is $\theta >0$ depending only on $\Lambda $ such that $\lvert N_{\delta}(2^{\frac{i+1}{2n}}) \rvert \ll_{\Lambda }2^{i/2- i\theta} \ll_\Lambda 2^{i/2}Z^{-\theta}$ since $2^{i} \geq Z$. The number of non-negative integers $i$ such that $Z \leq 2^i \leq L^{2n}$ is clearly bounded by $L^\varepsilon$ for any $\varepsilon>0$ so the expression in \eqref{squarefull_values} is bounded by $L^{n+ \varepsilon}Z^{- \theta}$ for any $\varepsilon >0$, and the proof is complete.
\end{proof}

\begin{remark}
In principle, Cohen's theorem \cite[Theorem 2.5]{cohen} allows us to make the exponent $\theta$ in the above proposition explicit in terms of the lattice $\Lambda$, but we did not find it very enlightening to do so.
\end{remark}

\section{Sums of type I}

In this section, we prove Proposition \ref{sums_of_type_I}. Fix a non-zero integral ideal $\mathfrak{M}$ and an element $\mu \in(\mathcal{O}_{F'}/ \mathfrak{M} \mathcal{O}_{F'} )^\times$. For each $i \in \left\{ 1,..., h_0 \right\}$, our task is to estimate
\begin{equation} \label{sumaftype1}
\sum_{\substack{
N(\mathfrak{a}) \leq X\\ 
\mathfrak{m} \mid \mathfrak{a}}} r_i(\mathfrak{a} , \mathfrak{M}, \mu )s_{\mathfrak{a}}
\end{equation}
when $\mathfrak{m}$ is a non-zero integral ideal of $\mathcal{O}_{F'}$. The first step is to reduce \eqref{sumaftype1} to a sum over principal ideals with generators having a fixed value modulo a modulus that we now define. Recall that $\mathfrak{p}_1 ,..., \mathfrak{p}_{h_0}$ denote prime ideals of degree $1$ over $\mathbb{Q}$ coprime to $3$ representing the subgroup $H_0 \leq H_{F'}(\mathbf{m} )$. Here the modulus $\mathbf{m}$ defined in Section \ref{the_spin_symbol} should not be confused with the ideal $\mathfrak{m}$ in the above sum. If $h$ denotes the class number of $K'$, we also fix prime ideals $\mathfrak{P}_1 ,..., \mathfrak{P}_h  $ of degree $1$ over $\mathbb{Q}$, coprime to $3$, not lying over any of $\mathfrak{p}_1 ,..., \mathfrak{p}_{h_0}$ and representing the ideal classes of $K'$. We now define
\begin{equation*}
F_0 := 2^4 \cdot 3^{3h+3} \cdot \Delta(K'/ \mathbb{Q}) \cdot N_{F'/ \mathbb{Q}}(\mathfrak{M} ) \cdot \prod_{i=1}^{h_0} N_{F'/\mathbb{Q} '}(\mathfrak{p}_i) \cdot  \prod_{j=1}^{h} N_{K'/\mathbb{Q}}(\mathfrak{P}_j )
\end{equation*}
where $\Delta(K'/ \mathbb{Q})$ denotes the absolute discriminant of the extension $K'/\mathbb{Q}$. 

\begin{lemma}\label{reduce_to_principal_ideals}
Fix an index $ i \in \left\{ 1,..., h_0 \right\}$ and $\rho \in \mathcal{O}_{F'}$. Suppose $\mathfrak{a} \in H_0 $ and $\mathfrak{a}\mathfrak{p}_i = (\alpha)$ for some $\alpha$ such that $\alpha \equiv \rho \bmod F_0$. Then 
\[
s_{\mathfrak{a}} = \gamma_{\rho,i} s_{(\alpha)}
\]
for some $\gamma _{\rho,i} \in \{1, \zeta_3, \zeta_3^2, 0 \}$ depending only on $\rho$ and $i$.
\end{lemma}

\begin{proof}
Since $(\alpha) = \mathfrak{a} \mathfrak{p}_i \in I_0$, $s_{(\alpha)}$ is defined. Let $N_{F'/F}(\mathfrak{p}_i)=(\pi_i)$. By using multiplicativity of the cubic residue symbol, we find that
\begin{equation*}
s_{(\alpha)}=s_{\mathfrak{a}} \left( \frac{N_{K/F}(\sigma(\pi_i))}{\mathfrak{a}} \right)_{3,F'} \left(\frac{N_{K/F}(\sigma(N_{F'/F}(\alpha)))}{\mathfrak{p}_i} \right)_{3,F'}.
\end{equation*}
Since $\mathfrak{p}_i$ divides the rational integer $F_0$ in $\mathcal{O}_{F'}$, it follows that the third factor is determined by $\alpha$ modulo $F_0$. The middle factor can be written as
\begin{equation*}
\left( \frac{N_{K/F}(\sigma(\pi_i))}{\alpha} \right)_{3,F'} s_{\mathfrak{p}_i}^{-1},
\end{equation*}
and we would like to show that the first of these factors depends only on $\rho$ and $i$. By Proposition \ref{cubic_reciprocity}, it depends only on $\alpha$ modulo $27 N_{K/F}(\sigma(\pi_i))$. By Lemma \ref{canonical},
\begin{equation*}
N_{K/F}(\sigma(\pi_i))= \sigma_1(\pi_i)\sigma_2(\pi_i) \sigma_3(\pi_i)=N_{K/L}(\pi_i)/\pi_i 
\end{equation*}
where $\sigma_1, \sigma_2 ,\sigma_3$ are the three double transpositions, and $L$ is the cubic subfield of $K$ fixed by $\sigma_1 , \sigma_2 , \sigma_3$. Hence the above cubic residue symbol depends only on $\alpha$ modulo $27 N_{K/L}(\pi_i)$. The number $N_{K/L}(\pi_i)$ is a rational integer since it is invariant under $\mathrm{Gal}(K/F)$ (as $\pi_i \in F$ and $\left\{ 1,\sigma_1 , \sigma_2 ,\sigma_3 \right\}$ is normal in $\mathrm{Gal}(K/\mathbb{Q})$). Using transitivity of norms,
\begin{equation*}
N_{F/ \mathbb{Q} }(\pi_i)^{\left[ K:F \right]}=N_{K/\mathbb{Q} }(\pi_i)= N_{L/\mathbb{Q} }(N_{K/L}(\pi_i))=N_{K/L}(\pi_i)^{\left[ L:\mathbb{Q} \right]}.
\end{equation*}
Since $\left[ K:F \right]= \left[ L:\mathbb{Q} \right]=3$, $N_{K/L}( \pi_i) = N_{F/ \mathbb{Q} }( \pi_i) = \pm N(\mathfrak{p}_i)$. Hence it only depends on $\alpha$ modulo $27 N(\mathfrak{p}_i)$. This number divides $F_0$ so we are done.     
\end{proof}

Next, we need an integral basis $\eta = \{\eta_1,..., \eta_8\}$ for $\mathcal{O}_{F'}$ with $\eta_1 =1$. It will be convenient to make our choice more specific. Fix an integral basis $1,\theta_2, \theta_3 , \theta_4$ for $\mathcal{O}_{F}$. Since $1, \zeta_3$ is an integral basis for $\mathcal{O}_{\mathbb{Q}(\zeta_3)}$, and the discriminants of $\mathbb{Q}(\zeta_3)$ and $F$ are coprime, it follows that
\begin{equation*}
1,\, \zeta_3,\, \theta_2,\, \theta_2 \zeta_3,\,\theta_3,\, \zeta_3 \theta_3 ,\, \theta_4,\, \zeta_3 \theta_4 
\end{equation*}
is an integral basis for $\mathcal{O}_{F'}$, which we label $\eta_1 ,..., \eta_8$. Let $\mathcal{D}$ be the fundamental domain given by Lemma \ref{fundamental_domain}. By the same argument as in \cite[p. 15]{peter} using Lemma \ref{reduce_to_principal_ideals}, it suffices to estimate, for a fixed $\rho$ modulo $F_0$, the sum
\[
A(X, \rho) =\sum_{\substack{ \alpha \in \mathcal{D}(X) \\ \alpha \equiv \rho \bmod F_0 \\ \alpha \equiv 0 \bmod \mathfrak{m}}} s_{(\alpha)}.
\]
We now manipulate $A(X,\rho)$. To this end, we use the expression for the spin symbol in \eqref{alternative_expression}. There is a decomposition $\mathcal{O}_{F'}=\mathbb{Z} \oplus \mathbb{M}$ where $\mathbb{M}= \mathbb{Z} \eta_2 \oplus \cdots \oplus \mathbb{Z} \eta_8$ so every $\alpha \in \mathcal{O}_{F'}$ can be written uniquely as $a+\beta$ where $a \in \mathbb{Z}$ and $\beta \in \mathbb{M}$.

We write $x\mapsto \overline{x}$ for the unique automorphism of $K(\zeta_3)$ that fixes $K$ and satisfies $\overline{\zeta_3}= \zeta_3^{-1}$. Any $\sigma \in \Gal(K/\mathbb{Q})$ has unique extension to $\Gal(K'/\mathbb{Q})$ which commutes with $x\mapsto \overline{x}$ which we also denote by $\sigma$. Moreover, for each $\sigma \in \mathrm{Gal}(K/ \mathbb{Q})$, we write $\overline{\sigma}$ for the automorphism $x\mapsto\sigma(\overline{x})$. Thus
\begin{equation*}
\begin{split}
  s_{(\alpha)}&= \left( \frac{\sigma(\alpha\overline{\alpha})}{ \alpha} \right)_{3,K'}\\
              &= \left( \frac{\sigma(\beta)+a}{a+\beta} \right)_{3,K'} \left( \frac{\overline{\sigma}(\beta)+ a}{a+\beta} \right)_{3,K'}\\
             & = \left( \frac{\sigma(\beta)-\beta}{a+\beta} \right)_{3,K'} \left( \frac{\overline{\sigma}(\beta)-\beta}{a+\beta} \right)_{3,K'}.
\end{split}
\end{equation*}

\noindent
We now consider $\beta$ to be fixed and let $a$ vary. Moreover, for the rest of the paper, we fix $\sigma$ to be one of the three double transpositions in $\mathrm{Gal}(K'/\mathbb{Q}(\zeta_3))$, and let $E$ be the fixed field of $\sigma$ (so $[K':E]=2$). If $\sigma(\beta)-\beta=0$ or $\overline{\sigma}(\beta)-\beta=0$ then $s_{(\alpha)}=0$ so we may remove these $\beta$ from consideration and assume $\sigma(\beta)-\beta,\sigma(\overline{\beta})-\beta\neq0$. By the same argument as in \cite[p. 726]{FIMR}, we have

\begin{lemma}\label{c_and_c_prime}
Assume $\alpha\equiv \rho\pmod{F_0}$, and write $\alpha=a+\beta$ for some $a\in\mathbb{Z}$ and $\beta\in\mathbb{M}$. Let $\mathfrak{c}$ and $\mathfrak{c}'$ be the greatest divisors of $\sigma(\beta)-\beta$ and $\overline{\sigma}(\beta)-\beta$ respectively coprime to $F_0$. Then
\begin{equation*}
\left( \frac{\sigma(\beta)-\beta}{a + \beta} \right)_{3,K'}=\mu \left( \frac{a+ \beta}{\mathfrak{c}} \right)_{3,K'}\textrm{  and  }\left(\frac{\sigma(\overline{\beta})-\beta}{a+\beta}\right)_{3,K'}=\mu'\left(\frac{a+\beta}{\mathfrak{c}'}\right)_{3,K'}
\end{equation*}
where $\mu,\mu' \in \left\{ 1,\zeta_3 , \zeta_3^2,0 \right\}$ depend on $\rho$ and $\beta$ but not on $a$. 
\end{lemma}

The proof uses the representatives $\mathfrak{P}_1 ,..., \mathfrak{P}_h  $ for ideal class group of $K'$ and the fact that $3^{3 h +3} \mid F_0$. We can now write
\begin{equation*}
A(X, \rho)= \sum_{\beta \in \mathbb{M}} \mu_{\rho, \beta} T(X, \rho, \beta)
\end{equation*}
where $\mu_{\rho, \beta} \in \left\{ 1, \zeta_3 , \zeta_3^2,0 \right\}$ depends on $\rho$ and $\beta$, and where
\begin{equation*}
T(X,\rho,\beta)= \sum_{\substack{
    a \in \mathbb{Z}\\
    a+ \beta \in \mathcal{D}(X)\\
    a+ \beta \equiv 0 \bmod{\mathfrak{m}}\\
    a+ \beta \equiv \rho \bmod{F_0}
}}\left( \frac{a+ \beta}{\mathfrak{c}} \right)_{3,K'} \left(\frac{a+\beta}{\mathfrak{c}'}\right)_{3,K'} 
\end{equation*}
where $\mathfrak{c}$ and $\mathfrak{c}'$ are as in Lemma \ref{c_and_c_prime}.  This has the same shape as the last equation in \cite[p. 14]{peter}, and we now perform a field lowering argument as in \cite[p. 15]{peter}. Unlike in this reference, the argument will not make our result unconditional. Instead, it allows to assume conjecture $C_{12}$ instead of Conjecture $C_{24}$. It will also play a much more important role in making the use of Proposition \ref{squarefull} possible. Without field lowering, we would have to count elements in a rank $6$ lattice inside a degree $24$ number field which seems completely out of reach with current methods. The reason that field lowering plays a new role in our argument is that the spin symbol is defined over a non-Galois extension.

We start by proving a result similar to \cite[Lemma 2.4]{peter}.

\begin{lemma}\label{field_lowering_lemma}
We have 
\begin{equation*}
\left( \frac{a+\beta}{\mathfrak{c} '} \right)_{3,K'} = \mathbf{1}_{\gcd(a+\beta,\mathfrak{c}')=1}.
\end{equation*}
\end{lemma}

\begin{proof}
If we let $E_0$ denote the fixed field of $\overline{\sigma}$, then $\mathfrak{c} '$ is an extension of an ideal of $\mathcal{O}_{E_0}$ since the ideal $(\overline{\sigma}(\beta)-\beta)$ is invariant under $\overline{\sigma}$. By an abuse of notation, we now consider $\mathfrak{c} '$ as an ideal of $\mathcal{O}_{E_0}$. Factor $\mathfrak{c} ' = \prod_{i=1}^k P_i^{e_i}$ where the $P_i$ are non-zero prime ideals of $\mathcal{O}_{E_0}$. Then
\begin{equation*}
  \left( \frac{a+\beta}{\mathfrak{c} ' \mathcal{O}_{K'}}\right)_{3,K'}= \prod_{i=1}^k \left( \frac{a +\beta}{P_i \mathcal{O}_{K'}} \right)_{3,K'}^{e_i}.
\end{equation*}
Fixing $P_i$, we have $\overline{\sigma}(a+\beta)\equiv a+\beta \bmod{P_i}$ so by \cite[Lemma 3.4]{peter}, there is $\beta ' \in \mathcal{O}_{E_0}$ such that $a+\beta \equiv \beta ' \bmod{P_i \mathcal{O}_{K'}}$ so
\begin{equation*}
\left( \frac{a+\beta}{P_i \mathcal{O}_{K'}} \right)_{3,K'}= \left( \frac{\beta '}{P_i \mathcal{O}_{K'}} \right)_{3,K'}.
\end{equation*}
By \cite[Lemma 2.4]{peter}, this equals $\mathbf{1}_{\gcd(a+\beta,P_i)=1}$ when $P_i$ splits completely in $K'$. If $P_i$ is inert in $K'$, we would like to appeal to \cite[Lemma 2.5]{peter}, but then we must assume that $P_i$ has degree $1$ over $\mathbb{Q}$ which is not always the case. We now explain how this assumption can be removed in our setting. Let $p$ be the prime in $\mathbb{Q}$ lying under $P_i$. We show that $N(P_i) \equiv 2 \bmod{3}$, and then everything is a cube modulo $P_i$, so the above cubic residue symbol also equals $\mathbf{1}_{\gcd(a+\beta,P_i)=1}$.\\

\noindent
Since $\mathfrak{c} '$ is coprime to $F_0$, $p$ is unramified in $K'$. Hence the inertial degree of $p$ in $K'$ is the order of any Frobenius element over $p$ in $\mathrm{Gal}(K'/\mathbb{Q})$. Since $\mathrm{Gal}(K'/\mathbb{Q} )\cong A_4 \times \left\{ \pm 1 \right\}$, an element can have order either $1$, $2$, $3$ or $6$. Since $P_i$ has inertial degree $2$ in $K'$, this is only possible if the inertial degree $f_{P_i /p}$ in $E_0$ is $1$ or $3$ so $N(P_i)= p$ or $N(P_i)=p^3$. In either case $N(P_i)\equiv p \bmod{3}$ so we must show $p \equiv 2 \bmod{3}$, or equivalently that $p$ is inert in $\mathbb{Q}(\zeta_3)$. Suppose for the sake of contradiction that $p$ splits in $\mathbb{Q}(\zeta_3)$. If $f_{P_i /p}=3$, then it follows that the two primes above $p$ in $\mathbb{Q}(\zeta_3)$ have inertial degree $6$ in $K'$. This is impossible since $\mathrm{Gal}(K'/\mathbb{Q}(\zeta_3))\cong A_4$ contains no elements of order $6$. Hence $f_{P_i /p}=1$ so $p$ has inertial degree $2$ in $K'$, and the decomposition group of $P_i \mathcal{O}_{K'}$ has size $2$ and equals $\mathrm{Gal}(K'/E_0)= \left\{ 1,\overline{\sigma} \right\}$. It follows that $\overline{\sigma}$ is a Frobenius element over $p$ in $\mathrm{Gal}(K'/\mathbb{Q})$. The inertial degree of $p$ in $\mathbb{Q}(\zeta_3)$ is now equal to the order of $\overline{\sigma} \mid_{\mathbb{Q}(\zeta_3)}$. Since $\overline{\sigma}(\zeta_3)\neq \zeta_3$ this order is $2$. This is a contradiction since we assume $p$ has inertial degree $1$ in $\mathbb{Q}(\zeta_3)$. So, in the first place, $p$ must have had inertial degree $2$ in $\mathbb{Q}(\zeta_3)$ which is to say $p \equiv 2 \bmod{3}$.
\end{proof}

\noindent
We now handle the other residue symbol $\big(\frac{a+\beta}{\mathfrak{c}}\big)_{3,K'}$. Here the analysis is more similar to \cite{peter}. Recall that $E$ is the fixed field of $\sigma$. Since $\sigma$ preserves the ideal $(\sigma(\beta)-\beta)$, it follows that $\mathfrak{c}$ is an extension of an ideal of $\mathcal{O}_{E}$ which we also denote $\mathfrak{c}$. We can factor $\mathfrak{c} = \mathfrak{g} \mathfrak{q}$ in $\mathcal{O}_E$ where $\mathfrak{g}$ has squarefull norm, $\mathfrak{q}$ has squarefree norm, and $\gcd(N(\mathfrak{g}),N(\mathfrak{q}))=1$. In particular, every prime factor of $\mathfrak{q}$ has degree $1$ over $\mathbb{Q}$. Arguing precisely as in \cite[Section 2.4]{peter}, we find that
\begin{equation*}
T(X, \rho , \beta)= \sum_{\substack{
    a \in \mathbb{Z}\\
    a+ \beta \in \mathcal{D}(X)\\
    a+\beta \equiv 0 \bmod{\mathfrak{m}}\\
    a+\beta \equiv \rho \bmod{F_0}
}}\left(\frac{a+ \beta}{\mathfrak{g} \mathcal{O}_{K'}} \right)_{3,K'}\cdot \left( \frac{a+b}{\mathfrak{q}} \right)_{3,E}^2 \cdot \mathbf{1}_{\gcd (a+\beta, \mathfrak{c} ')=1}
\end{equation*}
where $b$ is a rational integer only depending on $\beta$.

Let $\mathfrak{g}_0 := \prod_{\mathfrak{p} \mid \mathfrak{g}} \mathfrak{p}$ denote the radical of $\mathfrak{g}$, and $g_0: = N_{E/ \mathbb{Q} }(\mathfrak{g}_0)$ the absolute norm of $\mathfrak{g}_0$. Then $\big(\frac{a + \beta}{\mathfrak{g} \mathcal{O}_{K'}}\big)_{3,K'}$ depends only on the residue class of $a$ modulo $\mathfrak{g}_0$. Hence
\begin{equation*}
\lvert T(X, \rho , \beta) \rvert \leq \sum_{a_0 \bmod{g_0}} \lvert T(X, \rho, \beta, a_0) \rvert
\end{equation*}
where
\begin{equation*}
T(X, \rho, \beta , a_0) := \sum_{\substack{
    a \in \mathbb{Z}\\
    a+ \beta \in \mathcal{D}(X)\\
    a+ \beta \equiv 0 \bmod{\mathfrak{m}}\\
    a+\beta \equiv \rho \bmod{F_0}\\
    a \equiv a_0 \bmod{g_0}
}} \left( \frac{a+b}{\mathfrak{q}} \right)_{3,E} \cdot \mathbf{1}_{\gcd(a+\beta,\mathfrak{c} ')=1}.
\end{equation*}
Note that we are not squaring the cubic residue symbol because this is equivalent to conjugating it which does not change the modulus of $T(X, \rho, \beta , a_0)$. Using Mobius inversion, we have
\begin{equation*}
\lvert T(X, \rho , \beta , a_0) \rvert \leq \sum_{\substack{
    \mathfrak{d} \mid \mathfrak{c} '\mathcal{O}_{K'}\\
    \mathfrak{d} \textrm{ squarefree}
}} \lvert T(X,\rho,\beta, a_0 , \mathfrak{d}) \rvert
\end{equation*}
where 
\begin{equation*}
T(X, \rho , \beta , a_0 , \mathfrak{d}):= \sum_{\substack{
    a \in \mathbb{Z}\\
    a + \beta \in \mathcal{D}(X)\\
    a+ \beta \equiv 0\bmod{\mathfrak{m}}\\
    a+\beta \equiv \rho \bmod{F_0}\\
    a \equiv a_0 \bmod{g_0}\\
    a+ \beta \equiv 0 \bmod{\mathfrak{d}}}} \left( \frac{a+b}{\mathfrak{q}} \right)_{3,E}.
\end{equation*}
\noindent
Let $q:= N_{E/ \mathbb{Q}}(\mathfrak{q})$ be the absolute norm of $\mathfrak{q}$ which is a squarefree integer coprime to $3$. The function $ \chi_{\mathfrak{q}}: \ell \mapsto \big(\frac{\ell}{\mathfrak{q}}\big)_{3,E}$ is Dirichlet character modulo $q$, and by Lemma \ref{non_principal}, it is non-principal for $q>1$. The summation conditions in $T (X,\rho,\beta,a_0,\mathfrak{d})$, means that $a$ runs over at most $8$ intervals of length $\ll X^{1/8}$ whose endpoints depend on $\beta$, and $a$ lies in an arithmetic progression of modulus $k$ dividing $m g_0 d F_0$ where $m:=N(\mathfrak{m})$ and $d:=N(\mathfrak{d})$. We claim that the modulus $q$ of $\chi_{\mathfrak{q}}$ is $\ll X^{\frac{3}{2}}$. First, observe that
\begin{equation*}
q=N_{E/  \mathbb{Q}}(\mathfrak{q})=N_{K' / \mathbb{Q}}(\mathfrak{q}\mathcal{O}_{K'})^{\frac{1}{2}} \leq |N_{K'/\mathbb{Q}}(\sigma(\beta)-\beta)|^{\frac{1}{2}}
\end{equation*}
and recall that we chose the integral basis $1=\eta_1,\eta_2,...,\eta_8$ for $\mathcal{O}_{F'}$, and $\beta=\sum_{i=2}^8a_i \eta_i$ for some integers $a_i$ satisfying $|a_i|\ll X^{\frac{1}{8}}$. The polynomial 
\begin{equation*}
F(X_2,...,X_8):=N_{K'/\mathbb{Q}}\left(\sum_{i=2}^8 X_i (\sigma(\eta_i)-\eta_i)\right)
\end{equation*}
has rational coefficients and total degree $24$. The coefficients depend only on $\eta_2,...,\eta_8$ and $\sigma$ which are all fixed. Hence 
\begin{equation*}
|N_{K'/\mathbb{Q}}(\sigma(\beta)-\beta)|=|F(a_2,...,a_8)|\ll (X^{1/8})^{24}=X^3,
\end{equation*}
and indeed $q\ll X^{\frac{3}{2}}$. Assuming Conjecture $C_n$ with $n=12$, Corollary \ref{short_character_sum_corollary} tells us that when $q\nmid k$, there is $\delta>0$ such that $T(X,\rho,\beta) \ll_{\varepsilon} g_0 X^{\frac{1}{8}-\delta+\varepsilon}$ for all $\varepsilon>0$. This is of course only interesting when $g_0$ is at most some small power of $X$. When $g_0$ is large, say $g_0 \geq Z$ for some fixed $Z$, we instead take a step back and use the estimate
\begin{equation*}
A(X,\rho) \ll X^{\frac{1}{8}} \# \left\{ \beta \in \mathbb{M}\,:\, \lvert a_i \rvert \ll X^{\frac{1}{8}},\, g_0 \geq Z \right\}.
\end{equation*}
Choosing $Z$ to be a small power of $X$, a power saving can be achieved using Proposition \ref{squarefull}.\\

\noindent
When the modulus $q$ of the Dirichlet character $\chi_{\mathfrak{q}}$ divides the modulus $k$ of the arithmetic progression there is of course no cancellation. We instead reformulate this as a condition similar to \cite[Equation (4.4)]{imrn}. Recall that $E$ and $E_0$ denote the fixed fields of $\sigma$ and $\overline{\sigma}$ respectively. Since $\sigma$ has order $2$, it follows that there is a non-zero $\alpha_0 \in \mathcal{O}_K$ such that $\sigma(\alpha_0)=-\alpha_0$, and since $\overline{\alpha_0}=\alpha_0$, we also have $\overline{\sigma }(\alpha_0)=-\alpha_0$. It follows that the images of the linear maps
\begin{equation*}
\phi : \mathbb{M} \rightarrow K',\quad x \mapsto \alpha_0(\sigma(x)-x)
\end{equation*}
\begin{equation*}
\phi' : \mathbb{M} \rightarrow K', \quad x \mapsto \alpha_0( \overline{\sigma}(x)-x)
\end{equation*}
are contained in $\mathcal{O}_{E}$ and $\mathcal{O}_{E_0}$ respectively. Both $\phi(\mathbb{M})$ and $\phi '(\mathbb{M})$ are lattices, and later we show that their ranks are $6$ and $7$ respectively. Returning to the situation when $q \mid k$, we prove the following:

\begin{lemma}\label{rephrasing}
Assume Conjecture $C_{12}$. Then there exists $\delta >0$ with the following property: If $\beta \in \mathbb{M}$, and $a+\beta \in \mathcal{D}(X)$ for some $a \in \mathbb{Z}$ then one of the following conditions hold
\begin{enumerate}[(i)]
\item $T(X, \rho , \beta) \ll_\varepsilon g_0 X^{\frac{1}{8}-\delta + \varepsilon}$ for all $\rho$ modulo $F_0$ and all $\varepsilon >0$.
\item For all primes $p$ we have the implication
\begin{equation}\label{firkant} \tag{$\square$}
p \mid N_{E/\mathbb{Q}}(\phi(\beta)) \Rightarrow p^2 \mid m F_0 N_{K'/ \mathbb{Q} }(\alpha_0) N_{E/\mathbb{Q} }(\phi(\beta))N_{E_0 / \mathbb{Q} }(\phi '(\beta)).
\end{equation}
\end{enumerate}
\end{lemma} 

\begin{proof}
We must prove that \eqref{firkant} holds if the modulus $q$ of $\chi_{\mathfrak{q}}$ divides the modulus $k$ of the arithmetic progression. In this case, $q \mid mdF_0$ where $d$ is the norm of a squarefree ideal dividing $\mathfrak{c} ' \mathcal{O}_{K'}$. Assume that $p \mid N_{E/ \mathbb{Q}}(\phi(\beta))$. We have
\begin{equation}\label{norms}
N_{E/ \mathbb{Q}}(\phi(\beta))^2 =N _{K'/\mathbb{Q}}(\phi(\beta))= N_{K'/ \mathbb{Q}}(\alpha_0)N_{K'/ \mathbb{Q}}(\sigma(\beta)-\beta)
\end{equation}
so if $p \mid N_{K/ \mathbb{Q}}(\alpha_0)$, then \eqref{firkant} holds for $p$. Otherwise, we find that $p \mid N_{K'/\mathbb{Q}}(\sigma(\beta)-\beta)$. If $p \mid F_0$, it is again clear that \eqref{firkant} holds so assume $p \nmid F_0$. Then $p \mid N_{E/\mathbb{Q}}(\mathfrak{c})$ since, by definition, $\mathfrak{c} \mathcal{O}_{K'}$ is the largest divisior of $(\sigma(\beta)-\beta)$ coprime to $F_0$. We have decomposed $\mathfrak{c} = \mathfrak{g} \mathfrak{q}$ where $N_{E/\mathbb{Q}}(\mathfrak{g})$ is squarefull and $N_{E/\mathbb{Q}}(\mathfrak{q})$ is squarefree. If $p \mid N_{E/ \mathbb{Q}}(\mathfrak{g})$ then $p^2 \mid N_{E/\mathbb{Q}}(\mathfrak{g}) $ so $p^2\mid N_{E/ \mathbb{Q}}(\mathfrak{c})$. Since $N_{E/ \mathbb{Q}}(\mathfrak{c})^2 \mid N_{K'/\mathbb{Q}}(\sigma(\beta)-\beta)$ it follows from \eqref{norms} that $p^2 \mid N_{E/ \mathbb{Q}}(\phi(\beta)) $ so \eqref{firkant} holds. Otherwise, $p \mid N_{E/\mathbb{Q}}(\mathfrak{q})=q$ so $p \mid md F_0$. The number $d$ is the norm of a squarefree divisor of the ideal $(\overline{\sigma}(\beta)-\beta)$ in $\mathcal{O}_{K'}$. Like in \eqref{norms}, we have $N_{E_0 / \mathbb{Q}}(\phi'(\beta))^2 = N_{K'/ \mathbb{Q}}(\alpha_0 ) N_{K'/ \mathbb{Q}}(\overline{\sigma}(\beta)-\beta)$ so if $p \mid d$, we have $p \mid N_{E_0 / \mathbb{Q}}(\phi '(\beta))$. Hence $p \mid m F_0 N_{E_0 / \mathbb{Q}}(\phi '(\beta))$, and \eqref{firkant} still holds.
\end{proof}

\noindent
Motivated by this lemma, let $A_{\square}(X, \rho)$ denote the contribution to $A(X, \rho)$ from $\beta \in \mathbb{M}$ such that the condition \eqref{firkant} in the above lemma holds, and let $A_0(X, \rho)$ denote the contribution from the remaining $\beta \in \mathbb{M}$ where (i) holds. We estimate $A_{\square}(X,\rho)$ and $A_0(X, \rho)$ separately, and in both cases we need the material from Section \ref{counting_ideals}. As a preparation, we prove  

\begin{lemma}\label{preliminary}
Let $\Lambda_1 := \phi (\mathbb{M})\subset E$ and $\Lambda_2 := \phi '(\mathbb{M})\subset E_0$. Then
\begin{enumerate}
\item $\Lambda_1$ is a $\mathbb{Z}$-lattice of rank $6$;
\item $\Lambda_2$ is a $\mathbb{Z}$-lattice of rank $7$;
\item There is no $\alpha \in E^\times$ such that $\alpha \Lambda_1$ is contained in a proper subfield of $E$.
\end{enumerate}
\end{lemma}

\begin{proof}
\hfill{\phantom{}}
\begin{enumerate}
\item Recall that $\mathrm{Gal}(K'/\mathbb{Q}(\zeta_3)) \cong A_4$, and $\sigma$ is a fixed double transposition in $A_4$. $F'$ has index $3$ in $K'$, so it is the fixed field of some $3$-cycle $\tau \in A_4$. Since $\sigma$ and $\tau$ generate $A_4$, it follows that $E\cap F'=\mathbb{Q}(\zeta_3)$ so the kernel of $\phi : \mathbb{M} \rightarrow E$ is $\mathbb{Z} \zeta_3$, and the image is the rank $6$ module spanned by $\alpha_0(\sigma(\eta_3)-\eta_3)$,..., $\alpha_0(\sigma(\eta_8)-\eta_8)$.

\item Let $\tau$ be as in the proof of (i). We must show that $\phi ' : \mathbb{M} \rightarrow E_0$ is injective which is equivalent to $E_0 \cap F'= \mathbb{Q}$ which in turn is equivalent to $\overline{\sigma}$ and $\tau$ generating $\mathrm{Gal}(K'/ \mathbb{Q} )$. Since $A_4$ is generated by any double transposition and $3$-cycle, it follows that, as a subgroup of $\mathrm{Gal}(K'/ \mathbb{Q}) \cong A_4 \times \left\{ \pm 1 \right\}$, $\left\langle \overline{\sigma},\tau \right\rangle$ is either all of $A_4 \times \left\{ \pm 1 \right\}$ or of the form $\left\{(g,\varepsilon(g))\,:\, g \in A_4 \right\}$ for some surjective homomorphism $\varepsilon : A_4 \rightarrow \left\{ \pm 1\right\}$. But $A_4$ has no subgroup of index $2$ so no such $\varepsilon$ can exist.

\item A $\mathbb{Z}$-basis for $\Lambda_1$ is $\alpha_0(\sigma(\eta_3)-\eta_3),..., \alpha_0(\sigma(\eta_8)-\eta_8)$. As we have already observed, it is enough to show that
\begin{equation*}
\frac{1}{ \alpha_0(\sigma(\eta_3)-\eta_3)} \Lambda_1
\end{equation*}
is not contained in a proper subfield of $E$. Since $(\sigma(\eta_4)-\eta_4)/(\sigma(\eta_3)-\eta_3)= \zeta_3$, this subfield must contain $\mathbb{Q}(\zeta_3)$. Since $\mathrm{Gal}(K'/ \mathbb{Q}(\zeta_3)) \cong A_4$, and $E$ is the fixed field of a double transposition, it follows that the only possibility is that it is contained in $L'$, the unique cubic subfield of $K'/ \mathbb{Q}(\zeta_3)$. If this were the case, then
\begin{equation*}
\frac{\sigma(\eta_i)-\eta_i}{\sigma(\eta_3)-\eta_3} \in L ' \quad \textrm{for }i=1,2,...,8.
\end{equation*}
Since $\sigma$ fixes $L'$, and $K' = L'F'$, it follows that the map 
\begin{equation*}
\quad x \mapsto \frac{\sigma(x)-x}{\sigma(\eta_3)-\eta_3}
\end{equation*}
is a non-zero $L'$-linear map $K' \rightarrow L'$. Since $\dim_{L'} K'=4$, the kernel must have $L'$-dimension $3$. But the kernel is $E$ which has degree $2$ over $L'$ so this is absurd.
\end{enumerate}  
\end{proof}

\noindent
As the final ingredient, we need an estimate of how often the greatest common divisor of $N_{E/ \mathbb{Q} }(\phi (\beta) )$ and $N_{E_0 / \mathbb{Q}}(\phi '(\beta))$ is large. The lemma below is similar to \cite[Lemma 3.2]{Koymans}, but this reference does not apply in our case, and we need to modify the proof. It turns out that the proof is easier in our case because the relations $\sigma (\eta_2)-\eta_2=0$ and $\overline{\sigma}(\eta_2)-\eta_2 \neq 0$ make the things we need visibly true, and there will be no need to use the Galois action as in the proof of \cite[Lemma 3.2]{Koymans}.

Recall that any $\beta \in \mathbb{M}$ can be written uniquely as $a_2 \eta_2 +\cdots+a_8 \eta_8$ where $a_i \in \mathbb{Z}$. Moreover, if $a+ \beta \in \mathcal{D}(X)$ for some $a \in \mathbb{Z}$, then $\lvert a_i \rvert \ll  X^{\frac{1}{8}}$ for all $i=2,3,...,8$ where the implied constant only depends on $F'$ and the integral basis $\eta_1 , ..., \eta_8$. 

\begin{lemma}\label{gcd}
There is $\theta >0$ such that for $Z>0$, we have the following estimate for all $\varepsilon >0$:
\begin{equation*}
  \# \left\{ \beta \in \mathbb{M}\,:\, \lvert a_i \rvert \ll  X^{\frac{1}{8}},\, \gcd(N_{E/ \mathbb{Q}}(\phi(\beta)),N_{E_0 / \mathbb{Q}}(\phi '(\beta))) \geq Z \right\}
\end{equation*}
\begin{equation*}
 \ll_\varepsilon X^{\varepsilon}(X^{\frac{7}{8}}Z^{-\theta}+X^{\frac{3}{4}}+Z^4)
\end{equation*}  
\end{lemma}

\begin{proof}
We follow the same steps as in the proof of \cite[Lemma 3.2]{Koymans}. Define polynomials in $\mathbb{Z} \left[ X_2 , ..., X_8 \right]$ by
\begin{equation*}
F_1(X_2 ,..., X_8) := N_{E / \mathbb{Q} }( \alpha_0(\sigma(\eta_2)-\eta_2) X_2+\cdots + \alpha_0(\sigma(\eta_8)-\eta_8) X_8 )
\end{equation*}
\begin{equation*}
F_2(X_2 , ..., X_8):= N_{E_0/ \mathbb{Q}}(\alpha_0(\overline{\sigma}(\eta_2)-\eta_2)X_2+\cdots+\alpha_0(\overline{\sigma}(\eta_8)-\eta_8) X_8) 
\end{equation*}
If $\beta = a_2 \eta_2 +\cdots+ a_8 \eta_8$, we write $F_i(\beta)$ for $F_i(a_2, ..., a_8)$, and we set
\begin{equation*}
G(X,Z) := \# \left\{ \beta \in \mathbb{M}\,:\, \lvert a_i \rvert \ll X^{\frac{1}{8}},\,\gcd(F_1(\beta),F_2(\beta)) \geq Z \right\}.
\end{equation*}
We first observe that $F_1$ and $F_2$ are coprime over $\overline{\mathbb{Q}} \left[ X_2 , ..., X_8 \right]$. Indeed, $\sigma(\eta_2)-\eta_2 = 0$ as $\eta_2 = \zeta_3$ is fixed by $\sigma$. On the other hand, $\overline{\sigma }(\eta_2)-\eta_2= \zeta_3^{-1} - \zeta_3\neq 0$. We can factor both $F_1$ and $F_2$ into linear factors over $\overline{\mathbb{Q}}$ by writing the norms as products of Galois conjugates. We then see that each linear factor of $F_1$ has coefficient $0$ to $X_2$ whereas each linear factor of $F_2$ has a non-zero coefficient to $X_2$. Hence no linear factor of $F_1$ can be associate to a linear factor $F_2$ so they must be coprime over $\overline{\mathbb{Q}}$.

Next, we use this to count how often $F_1(\beta)$ and $F_2(\beta)$ have a large prime factor in common. Since $F_1$ and $F_2$ are coprime, we can argue as in the proof of \cite[Lemma 3.2]{Koymans} and use Bhargava's sieve \cite[Theorem 3.3]{bhargava} to see that for any $M>0$, we have
\begin{equation*}
\# \left\{ \beta \in \mathbb{M}\,:\,\lvert a_i \rvert  \leq X^{\frac{1}{8}},\,\exists \textrm{ prime } p \mid \gcd(F_1(\beta),F_2(\beta)),\,p>M\right\} \ll \frac{X^{\frac{7}{8}}}{M \log M} + X^{\frac{3}{4}}.
\end{equation*}
This gives a power saving when $M$ is a positive power of $X$.

For the remaining $\beta$, we factor $F_1(\beta)$ and $F_2(\beta)$ into the squarefull and squarefree parts: $F_i(\beta)=g_i q_i$ where $\gcd(g_i,q_i)=1$, $g_i$ is squarefull and $q_i$ is squarefree for $i=1,2$. By two applications of Proposition \ref{squarefull} with $M=E$ and $M=E_0$, Equation (\ref{normal_lattice_counting}) and Lemma \ref{preliminary} there is $\theta >0$ (depending only on $\eta_1 ,..., \eta_8$) such that 
\begin{equation*}
\# \left\{ \beta \in \mathbb{M}\,:\,\lvert a_i \rvert \ll X^{\frac{1}{8}},\, g_1 \geq A \textrm{ or }g_2 \geq A\right\} \ll_\varepsilon X^{\frac{7}{8}+\varepsilon}(A^{-\theta}+A^{-\frac{1}{7}})
\end{equation*}
for all $A>0$ and $\varepsilon >0$. Arguing in as in the proof of \cite[Lemma 3.2]{Koymans}, we have
\begin{equation*}
G(X,Z) \ll_\varepsilon \frac{X^{\frac{7}{8}}}{M \log M}+X^{\frac{3}{4}}+X^{\frac{7}{8}+\varepsilon}(A^{-\theta}+A^{-\frac{1}{7}})+G' \left( X, \frac{Z}{A^2},\frac{ZM}{A^2} \right)
\end{equation*}
where
\begin{equation*}
G' \left( X, \frac{Z}{A^2},\frac{ZM}{A^2} \right):= \# \left\{ \beta \in \mathbb{M}\,:\, \lvert a_i \rvert \ll X^{\frac{1}{8}},\,\exists r \mid \gcd(q_1 , q_2 ),\, \frac{Z}{A^2}<r \leq \frac{ZM}{A^2} \right\}.
\end{equation*}
We now estimate this quantity. Let $r$ be a squarefree integer. If $\mathfrak{r}_1$ and $\mathfrak{r}_2$ are ideals of $\mathcal{O}_{E}$ and $\mathcal{O}_{E_0}$ respectively with absolute norm $r$, we set
\begin{equation*}
E_{\mathfrak{r}_1 , \mathfrak{r}_2} := \# \left\{ \beta \in \mathbb{M}\,:\,\lvert a_i \rvert \ll  X^{\frac{1}{8}},\, \mathfrak{r}_1 \mid \phi(\beta),\, \mathfrak{r}_2 \mid \phi '(\beta)  \right\}
\end{equation*}
so that
\begin{equation*}
G' \left( X, \frac{Z}{A^2},\frac{ZM }{A^2} \right) \leq \sum_{\substack{
\frac{Z}{A^2} < r \leq \frac{ZM}{A^2}\\ r \textrm{ squarefree} 
}}\; \sum_{\substack{
\mathfrak{r}_1 , \mathfrak{r}_2 \\ N_{E/\mathbb{Q} }(\mathfrak{r}_1)=N_{E_0 / \mathbb{Q} }(\mathfrak{r}_2)=r 
}} E_{\mathfrak{r}_1 , \mathfrak{r}_2}.
\end{equation*}

When estimating $E_{\mathfrak{r}_1 , \mathfrak{r}_2}$, we need to be extra careful because, unlike in \cite{Koymans}, $\mathfrak{r}_1$ and $\mathfrak{r}_2$ are not ideals in the same ring. However, we can again take advantage of the fact that $\sigma(\eta_2)-\eta_2=0$ and $\overline{\sigma}(\eta_2)-\eta_2 \neq 0$. Split the coefficients $a_2 ,..., a_8$ according to their residue classes modulo $r$. Suppose $p$ is a prime factor of $r$, and let $\mathfrak{p}_1$ and $\mathfrak{p}_2$ be the unique prime factors of $\mathfrak{r}_1$ and $\mathfrak{r}_2$ respectively whose norm onto $\mathbb{Q} $ is $p$. Then for $\beta$ satisfying $\mathfrak{r}_1 \mid \phi(\beta)$ and $\mathfrak{r}_2 \mid \phi'(\beta)$, we have
\begin{equation*}
  \sum_{i=2}^{8} a_i(\alpha_0(\sigma(\eta_i)-\eta_i)) \equiv 0 \pmod{\mathfrak{p}_1}
\end{equation*}
\begin{equation*}
\sum_{i = 2}^8 a_i(\alpha_0(\overline{\sigma}(\eta_i)-\eta_i)) \equiv 0\pmod{\mathfrak{p}_2}
\end{equation*}
Since $\mathfrak{p}_1$ and $\mathfrak{p}_2$ have degree $1$ over $\mathbb{Q}$, we see that $(a_2 , ...,  a_8)$ satisfies a system of two linear equations over $\mathbb{F}_p$. We claim that when $\mathfrak{p}_1$ does not divide $\alpha_0(\sigma(\eta_3)-\eta_3)\neq 0$, and $\mathfrak{p}_2$ does not divide $\alpha_0(\overline{\sigma}(\eta_2)-\eta_2)\neq 0$, then the two equations are linearly independent over $\mathbb{F}_p$. Indeed, if this is the case, then since $\sigma(\eta_2)-\eta_2 =0$, the coefficient matrix takes the form
\begin{equation*}
\begin{pmatrix}
0&b_{12}&\cdots & b_{17}\\
b_{21}&b_{22}&\cdots & b_{27}\\
\end{pmatrix} \in M_{2 \times 7}(\mathbb{F}_p)
\end{equation*}
where $b_{12}$ and $b_{21}$ are non-zero elements of $\mathbb{F}_p$. This matrix clearly has full rank, so the equations are linearly independent. Hence there are $p^{7-2}=p^{5}$ possibilities for $a_2 ,..., a_8$ modulo $p$. For the prime dividing either of the two numbers above, the matrix can have rank $1$ or $0$, and we bound the number of solutions by $p^{7}$. Since there are only finitely many such primes, it follows by the Chinese remainder theorem that
\begin{equation*}
E_{\mathfrak{r}_1 , \mathfrak{r}_2} \ll r^{5} \left( \frac{X^{\frac{1}8{}}}{r}+1 \right)^7 \ll X^{\frac{7}{8}}r^{-2}+r^{5}.
\end{equation*}
We can now argue precisely as in \cite[p. 1745-1746]{Koymans} to achieve
\begin{equation*}
G' \left( X, \frac{Z}{A^2},\frac{ZM}{A^2} \right) \ll_\varepsilon X^{\varepsilon} \left( X^{\frac{7}{8}}\frac{A^2}{Z}+ \left( \frac{ZM}{A^2} \right)^{6} \right).
\end{equation*}
Taking $A=M=Z^{\frac{1}{3}}$ gives the desired bound. 
\end{proof}

\noindent
Estimating $A_{\square}(X, \rho)$ and $A_{0}(X, \rho)$ is now only a matter of putting everything together.

\begin{lemma}
There is $\vartheta >0$ such that $A_{\square}(X, \rho) \ll_\varepsilon X^{1-\vartheta + \varepsilon}$ for all $\varepsilon >0$.
\end{lemma}

\begin{proof}
For positive reals $Y$ and $Z$ to be determined, we write
\begin{equation*}
A_{\square}(X, \rho) = A_{\square}'(X,\rho)+A_{\square}''(X,\rho)+ A_{\square}'''(X,\rho)
\end{equation*}
where $A_{\square}'(X,\rho)$, $A''_{\square}(X, \rho)$ and $A '''_{\square} (X,\rho)$ denote the contribution from $\beta$ satisfying
\begin{itemize}
\item $\gcd(N_{E/ \mathbb{Q} }(\phi(\beta)), N_{E_0 / \mathbb{Q}}(\phi'(\beta))) <Z $ and $ \mathrm{sqfull}(N_{E/\mathbb{Q}}(\phi(\beta))) < Y$
\item $\gcd(N_{E/ \mathbb{Q} }(\phi(\beta)), N_{E_0 / \mathbb{Q}}(\phi'(\beta))) <Z$ and $ \mathrm{sqfull}(N_{E/\mathbb{Q}}(\phi(\beta))) \geq Y$
\item $\gcd(N_{E/\mathbb{Q} }(\phi(\beta)),N_{E_0 / \mathbb{Q} }(\phi '(\beta))) \geq Z$
\end{itemize}
respectively. By Proposition \ref{squarefull}, it follows that $A_{\square} ''(X,\rho) \ll_\varepsilon X^{1+\varepsilon}Y^{-\theta''}$ for some $\theta ''>0$, and by Lemma \ref{gcd}, it follows that there is $\theta '''>0$ such that $A_{\square}'''(X,\rho) \ll_\varepsilon X^{\varepsilon}(X^{\frac{7}{8}}Z^{-\theta'''}+X^{\frac{3}{4}}+Z^4)$. Hence we only need to estimate $A_{\square}'(X,\rho)$. Let $\beta$ satisfy the conditions defining $A_{\square}'(X,\rho)$, and write $N_{E/ \mathbb{Q}}(\phi(\beta))=gqr$ where
\begin{itemize}
\item $g$ is squarefull and coprime to $m N_{K/ \mathbb{Q}}(\alpha_0) F_0$;
\item $q$ is squarefree and coprime to $m N_{K/ \mathbb{Q}}(\alpha_0)F_0$;
\item $g$ and $q$ are coprime;
\item $r$ divides $(m N_{K/ \mathbb{Q} }(\alpha_0)F_0)^\infty$.
\end{itemize}
By assumption $g <Y$, and condition \eqref{firkant} in Lemma \ref{rephrasing} forces $q$ to divide $N_{E_0/\mathbb{Q} }(\phi '(\beta))$ so $q <Z$. Hence there are $\ll Y^{1/2}$ possibilities for $g$, and $\ll Z$ possibilities for $q$.  Since $m \leq X$, there are at most $ \ll_\varepsilon X^\varepsilon$ possibilities for $r$ for any $\varepsilon >0$. This leaves at most $\ll_\varepsilon X^\varepsilon Y^{1/2}Z$ possibilities for the value of $N_{E/ \mathbb{Q} }(\phi(\beta))$. Since $N_{E /\mathbb{Q}}(\phi(\beta)) \ll X^{\frac{3}{2}}$, it follows by Lemma \ref{schmidt} that $A'_{\square}(X,\rho) \ll_\varepsilon X^{\varepsilon}Y^{1/2}Z$ for any $\varepsilon>0$. Choosing $Y=X^{\delta_1}$ and $Z=X^{\delta_2}$ for $\delta_1$ and $\delta_2$ suitably small completes the proof. 
\end{proof}

We now estimate $A_{0}(X, \rho)$ and hence complete the proof of Proposition \ref{sums_of_type_I} 

\begin{lemma}
Assume Conjecture $C_{12}$. Then there is $\vartheta >0$ such that $A_0(X, \rho) \ll_\varepsilon X^{1-\vartheta + \varepsilon}$ for all $\varepsilon >0$.
\end{lemma}

\begin{proof}
Let $Z>0$, and write $A_0(X,\rho)= A_1(X, \rho)+ A_2(X, \rho)$ where $A_1(X, \rho)$ denotes the contribution from $\beta$ satisfying $\mathrm{sqfull}(N_{E / \mathbb{Q} }(\phi(\beta)))<Z$, and $A_2(X,\rho)$ is the contribution from the remaining $\beta$. When $\mathrm{sqfull}(N_{E/ \mathbb{Q} }(\phi(\beta))) <Z$, the number $g_0$ in condition (1) of Lemma \ref{rephrasing} satisfies $g_0 \ll Z$. Assuming Conjecture $C_{12}$, there is $\delta>0$ such that $T(X,\rho, \beta) \ll_\varepsilon  Z X^{\frac{1}{8}-\delta +\varepsilon}$ for each $\rho$ modulo $F_0$. The number of possibilities for $\beta$ is bounded by $X^{\frac{7}{8}}$, so it follows that $A_1(X,\rho) \ll_\varepsilon X^{1-\delta+\varepsilon}Z$. By Proposition \ref{squarefull}, $A_2(X,\rho) \ll_\varepsilon X^{1 +\varepsilon}Z^{-\theta}$ for some $\theta >0$. Choosing $Z=X^{\frac{\delta}{2}}$ completes the proof.
\end{proof}

\begin{remark}\label{gap}
In the course of writing this paper, we discovered a minor gap in \cite{imrn} where the field lowering technique was first introduced. The problem occurs on page 7423 when estimating the sum $A_{\square}(x;\rho, u_i)$, the analogue of our $A_{\square}(X,\rho)$. When bounding $A_{\square}(x;\rho,u_i)$ by a sum over integers $b$ satisfying the condition $p \mid b\Rightarrow p^2 \mid m d F b$, it is not taken into account that the integer $d$ depends on $b$. We found two other papers relying on this argument, \cite{peter} and \cite{piccolo}, but in all three cases the gap can be fixed by an argument that is very similar to ours.
\end{remark}

\section{Sums of type II}

We now prove Proposition \ref{sums_of_type_II}, taking the same approach as in \cite{peter}. For fixed $i \in \left\{ 1,..., h_0 \right\}$, $\mathfrak{M}\subset \mathcal{O}_{F'}$ and $\mu\in(\mathcal{O}_{F'}/\mathfrak{M}\mathcal{O}_{F'})^{\times}$, we must estimate
\begin{equation*}
\sum_{N(\mathfrak{m}) \leq M} \sum_{N(\mathfrak{n}) \leq N} v_{\mathfrak{m}} w_{\mathfrak{n}}r_i(\mathfrak{m} \mathfrak{n},\mathfrak{M},\mu) s_{\mathfrak{m} \mathfrak{n}}
\end{equation*}
where $(v_{\mathfrak{m}})_{\mathfrak{m}}$ and $(w_{\mathfrak{n}})_{\mathfrak{n}}$ are sequences of complex numbers of modulus at most $1$. Let $N_{F'/F}(\mathfrak{m}) =(m)$ and $N_{F'/F}(\mathfrak{n}) =(n)$ for some $m,n \in \mathcal{O}_{F}$. Then
\begin{equation*}
  s_{\mathfrak{m} \mathfrak{n}} = \left( \frac{\sigma(m)}{\mathfrak{m} \mathcal{O}_{K'}} \right)_{3,K'} \left( \frac{\sigma(n)}{ \mathfrak{n} \mathcal{O}_{K'}} \right)_{3,K'} \left( \frac{\sigma(m)}{\mathfrak{n} \mathcal{O}_{K'}} \right)_{3,K'} \left( \frac{\sigma(n)}{\mathfrak{m} \mathcal{O}_{K'}} \right)_{3,K'}.
\end{equation*}
The first two factors can be absorbed into $v_{\mathfrak{m}}$ and $w_{\mathfrak{n}}$. Moreover, $r_i(\mathfrak{m} \mathfrak{n},\mathfrak{M},\mu)= 1$ if and only if $r_k(\mathfrak{m},\mathfrak{M},\mu')=r_l(\mathfrak{n},\mathfrak{M},\mu'')=1$ for some $k,l \in \left\{ 1,...,h_0 \right\}$ such that $\mathfrak{p}_k \mathfrak{p}_l $ and $\mathfrak{p}_i$ represent the same class is $H_0$, and some $\mu',\mu''\in(\mathcal{O}_{F'}/\mathfrak{M}\mathcal{O}_{F'})^{\times}$. Hence it is enough to consider sums of the type
\begin{equation*}
\sum_{N(\mathfrak{m}) \leq M} \sum_{N(\mathfrak{n}) \leq N} v_{\mathfrak{m} } w_{\mathfrak{n}} r_k(\mathfrak{m},\mathfrak{M},\mu') r_l(\mathfrak{n},\mathfrak{M},\mu'') \left( \frac{\sigma(m)}{\mathfrak{n} \mathcal{O}_{K'}} \right)_{3,K'} \left( \frac{\sigma(n)}{ \mathfrak{m} \mathcal{O}_{K'}} \right)_{3,K'}
\end{equation*}
We can write $\mathfrak{m} \mathfrak{p}_k =(m')$ and $\mathfrak{n} \mathfrak{p}_l =(n ')$ for some $m ', n' \in \mathcal{O}_{F'}$. Replacing $m$ and $n$ with associate elements if necessary, we can assume that $m' \overline{m'}=\pi_k m$, and $n' \overline{n'}=\pi_l n$ where $\pi_k$ and $\pi_l$ are generators of $N_{F'/F}(\mathfrak{p}_k)$ and $N_{F'/F}(\mathfrak{p}_l)$ respectively. We now find that 
\begin{equation*}
\left( \frac{\sigma(m')\overline{\sigma}(m')}{n'} \right)_{3,K'} = \left( \frac{\sigma(m)}{\mathfrak{n} \mathcal{O}_{K'}} \right)_{3,K'} \left( \frac{\sigma(m)}{\mathfrak{p}_l \mathcal{O}_{K'}} \right)_{3,K'} \left(\frac{\sigma(\pi_k)}{n'}\right)_{3,K'},
\end{equation*}
and the last two factors depend only on $m'$ and $n'$ modulo $F_0$ by cubic reciprocity. Choosing $m'$ and $n'$ to lie in the fundamental domain $\mathcal{D}$, it follows that it suffices to estimate the sums
\begin{equation*}
\sum_{\substack{
\alpha \in \mathcal{D}(X)\\ \alpha \equiv \rho_1 \bmod{F_0} }} \sum_{\substack{
\beta \in \mathcal{D} (Y)\\ \beta \equiv \rho_2 \bmod{F_0 
}}} v_{\alpha} w_{\beta} \left( \frac{\sigma(\alpha) \overline{\sigma}(\alpha)}{\beta} \right)_{3,K'}\left( \frac{\sigma(\beta) \overline{\sigma}(\beta)}{\alpha} \right)_{3,K'}
\end{equation*}
where $\rho_1$ and $\rho_2$ are fixed residue classes modulo $F_0$. Since $\sigma$ fixes $\zeta_3$, it follows by cubic reciprocity that
\begin{equation*}
\left( \frac{\sigma(\alpha)}{\beta} \right)_{3,K'} = \left( \frac{\alpha}{\sigma(\beta)} \right)_{3,K'}= \mu_1 \left( \frac{\sigma(\beta)}{\alpha} \right)_{3,K'}
\end{equation*}
and
\begin{equation*}
\left( \frac{\overline{\sigma}(\alpha)}{ \beta} \right)_{3,K'} = \left( \frac{\alpha }{\overline{\sigma}(\beta)} \right)_{3,K'}^{-1} = \mu_2 \left( \frac{\overline{\sigma }(\beta)}{\alpha} \right)_{3,K'}^{-1} 
\end{equation*}
for some $\mu_1$ and $\mu_2$ only depending on $\rho_1$ and $\rho_2$, thus the problem reduces to estimating
\begin{equation*}
\sum_{\substack{
\alpha \in \mathcal{D}(X)\\ \alpha \equiv \rho_1 \bmod{F_0} }
} \sum_{\substack{
\beta \in \mathcal{D} (Y)\\ \beta \equiv \rho_2 \bmod{F_0} 
}} v_{\alpha} w_{\beta} \gamma(\alpha , \beta) \quad \textrm{where} \quad \gamma(\alpha , \beta)= \left( \frac{\sigma(\beta)}{\alpha} \right)_{3,K'}.
\end{equation*}
As in \cite{peter}, we can handle this expression using \cite[Proposition 4.3]{Koymans_Rome_2024}. We must specify a couple of parameters to use this result. Let $M $ denote the product of $F_0$ and the index $\left[\mathcal{O}_{K'}: \mathcal{O}_{F'} \sigma(\mathcal{O}_{F'}) \right]$. This index is finite because $K' = F' \sigma(F')$, so the extension $K'/\mathbb{Q} $ has a primitive element of the form $\sum a_i b_i$ where $a_i \in F'$ and $b_i \in \sigma(F')$. Multiplying this by a non-zero integer, we can assume $a_i$ and $b_i$ are integral so the index must indeed be finite. We also define $A_{\textrm{bad}}$ as the set of squarefull numbers. 

To get a non-trivial power saving, we just have to verify that $\gamma$ satisfies the properties (P1), (P2) and (P3) listed in \cite[p. 1314]{Koymans_Rome_2024}. (P1) is just the statement that $\gamma$ is multiplicative in each of its arguments which is clear in our case. (P3) asserts that $\# A_{\textrm{bad}}\cap [1,X]  \leq c_1 X^{1-c_2}$ for some absolute constants $c_1>0$ and $0 < c_2 <1$. It is well-known that such constants exist (we can take $c_2 =1/2$ and $c_1 $ some sufficiently large positive integer). Hence it only remains to prove that (P2) holds.\\

\noindent
To verify (P2), we must first show that if $w, z_1 , z_2 \in \mathcal{O}_{F'}$ are coprime to $M$ then $z_1 \equiv z_2 \pmod{MN_{F'/\mathbb{Q} }(w)}$ implies $\gamma(w,z_1) = \gamma(w, z_2)$. 
This is clear from the definition of $\gamma$. Finally, we must show that if $|N_{F'/ \mathbb{Q}}(w)| \notin A_{\textrm{bad}}$, then $z \mapsto \gamma(w,z)$ is a non-principal character on $(\mathcal{O}_{F'}/ M N_{F'/\mathbb{Q}}(w) \mathcal{O}_{F'})^\times$. The condition $N_{F'/\mathbb{Q}}(w) \notin A_{\textrm{bad}}$ means that the ideal $(w)$ has a degree $1$ prime factor $\mathfrak{p}$ dividing it exactly once. By the Chinese remainder theorem, it is enough to show that
\begin{equation*}
z \mapsto \left( \frac{\sigma(z)}{\mathfrak{p} \mathcal{O}_{K'}} \right)_{3,K'}
\end{equation*}
is a non-principal character on $(\mathcal{O}_{F'}/ \mathfrak{p} \mathcal{O}_{F'})^\times$. If $p := N_{F'/ \mathbb{Q}}(\mathfrak{p})$, the quotient $\mathcal{O}_{K'}/ \mathfrak{p} \mathcal{O}_{K'}$ is an $\mathbb{F}_p$-vector space of dimension $3$. Since the index $\left[\mathcal{O}_{K'}: \mathcal{O}_{F'} \sigma(\mathcal{O}_{F'}) \right]$ is finite and coprime to $p$, $\mathcal{O}_{F'}\sigma (\mathcal{O}_{F'})$ surjects onto $\mathcal{O}_{K'}/ \mathfrak{p} \mathcal{O}_{K'}$. Because $\mathfrak{p}$ has degree $1$ over $\mathbb{Q}$, $\mathcal{O}_{F'}$ and $\mathbb{Z}$ have the same image in $\mathcal{O}_{K'}/ \mathfrak{p} \mathcal{O}_{K'}$, and since $\sigma(\mathcal{O}_{F'})$ is already a $\mathbb{Z}$-algebra, it follows that $\mathcal{O}_{F'} \sigma(\mathcal{O}_{F'})$ and $\sigma(\mathcal{O}_{F'})$ have the same image in $\mathcal{O}_{K'}/ \mathfrak{p} \mathcal{O}_{K'}$, i.e. $\sigma(\mathcal{O}_{F'})$ surjects onto $\mathcal{O}_{K'}/ \mathfrak{p} \mathcal{O}_{K'}$. Now we are done because $p \equiv 1 \pmod{3}$, and $\mathfrak{p}$ is unramified in $K'$ because $p \nmid F_0$, so the residue symbol
\begin{equation*}
\left( \frac{\bullet}{\mathfrak{p} \mathcal{O}_{K'}} \right)_{3,K'}:(\mathcal{O}_{K'}/\mathfrak{p} \mathcal{O}_{K'})^\times  \rightarrow \left\{ 1, \zeta_3 ,\zeta_3^2 \right\}
\end{equation*}
is not identically equal to $1$.

We have now verified that (P1)-(P3) hold so by \cite[Proposition 4.3]{Koymans_Rome_2024} we have
\begin{equation*}
\sum_{\substack{
    \alpha \in \mathcal{D}(X)\\
    \alpha \equiv \rho_1 \bmod{N(\mathfrak{f}^\ast)}
}}\, \sum_{\substack{
  \beta \in \mathcal{D}(Y)\\
  \beta \equiv \rho_2 \bmod{N(\mathfrak{f}^\ast)}
  }} v_{\alpha} w_{\beta} \gamma(\alpha , \beta) \ll_\varepsilon(X+Y)^{\frac{1}{48}}(X Y)^{1-\frac{1}{48}+\varepsilon}
\end{equation*}
for any $\varepsilon >0$ where the implied constant depends only on the number field $F'$ and the constants $M$, $C_1$ and $C_2$. This is what we wanted.

\section{Level-raising of even Galois representations}
\label{proofs_for_even_representations}
\noindent 
In this section, we give the proof of Corollary \ref{selmerdensity}. First, some necessary preliminaries are recorded.\\ 

\noindent
The adjoint representation of $\SL_2(\F_3)$ on its Lie algebra of traceless matrices becomes a Galois module when composed with $\overline{\rho}$, denoted $\operatorname{Ad}^0(\overline{\rho})$. Let $\operatorname{Pl}_\Q$ be the set of places in $\Q$, including the archimedean place, $\infty$. For any $S \subseteq \operatorname{Pl}_\Q$, let $\Q_S \subseteq \overline{\Q}$ be the maximal extension of $\Q$ unramified at all $v\notin S$, and let $G_S := \Gal(\Q_S/\Q)$. 
For any $v\in \operatorname{Pl}_\Q$, let $G_v:=\Gal(\overline{\Q}_v/\Q_v)$.
For subspaces $\mathcal{N}_v \subseteq H^1(G_v, \operatorname{Ad}^0(\overline{\rho}))$, 
define the Selmer group 
\begin{equation} \label{sel}
H^1_{\mathcal{N}}(G_{S}, \operatorname{Ad}^0(\overline{\rho}))  =
\ker \left( 
H^1(G_S, \operatorname{Ad}^0(\overline{\rho})) \xrightarrow[]{} \bigoplus_{v\in S} H^1(G_v,\operatorname{Ad}^0(\overline{\rho}))/\mathcal{N}_v
\right)
\end{equation}
For the dual module $\operatorname{Ad}^0(\overline{\rho})^* = \operatorname{Hom}( \operatorname{Ad}^0(\overline{\rho}), \F_3)$,
denote
the annihilator of $\mathcal{N}_v$
under the local pairing as $\mathcal{N}_v^\perp$. The dual Selmer group is defined as 
\begin{equation} \label{dsel}
H^1_{\mathcal{N}^\perp}(G_{S}, \operatorname{Ad}^0(\overline{\rho})^*)
=
\ker \left(
H^1(G_S, \operatorname{Ad}^0(\overline{\rho})^*) \xrightarrow[]{} \bigoplus_{v\in S} H^1(G_v, \operatorname{Ad}^0(\overline{\rho})^*)/\mathcal{N}^\perp_v
\right)
\end{equation}
If
\begin{equation}\label{balancedglobal}
  \dim H^1_{\mathcal{N}}(G_{S}, \operatorname{Ad}^0(\overline{\rho}) )
  = 
  \dim H^1_{\mathcal{N}^\perp}(G_{S}, \operatorname{Ad}^0(\overline{\rho})^* ),
\end{equation}
we say that the 
\emph{global setting is balanced at $S$} and refer to 
the common rank of the Selmer and dual Selmer group as the rank of the global setting.\\

\noindent 
Next, we will choose the spaces $\mathcal{N}_v$
for each $v\in S$ in a particular way: 
Assume that the local versal deformation ring $R_v$ has a smooth quotient
\begin{equation} \label{smooth1}
R_v \to \Z_3[[T_1,\dots,T_{n_v}]]
\end{equation}
with tangent space $\mathcal{N}_v$ 
such that for any tame prime in $S$ we have 
\[
\dim \mathcal{N}_v = \dim H^0(G_v, \operatorname{Ad}^0(\overline{\rho}))
\]
and such that  
\[
\dim \mathcal{N}_3 +\dim \mathcal{N}_\infty 
=\dim H^0(G_3, \operatorname{Ad}^0(\overline{\rho}))+ \dim H^0(G_\infty, \operatorname{Ad}^0(\overline{\rho})). 
\]
For each $v\in S$, 
let $\mathcal{C}_v$ be the class of deformations of $\overline{\rho}|_{G_v}$ that factor through \eqref{smooth1}.
Consider an irreducible representation
\[
    \rho: \Gal(\overline{\Q}/\Q) \to \SL_2(\Z_3)
\]
unramified outside $S$ such that $\rho|_{G_v} \in \mathcal{C}_v$ for all $v\in S$. 
Let $\overline{\rho} := (\rho \bmod 3)$ and for any prime $p$ let $R_p$ be the local versal deformation ring at $p$.
A prime $p \notin S$ \emph{raises the level} of $\rho$ if  
there is a smooth quotient 
\begin{equation}\label{smoothp}
R_p \to \Z_3[[T_1,\ldots, T_{n_p}]]
\end{equation}
isomorphic to a power series ring over $\Z_3$ with 
\[
n_p = H^0(G_p, \operatorname{Ad}^0(\overline{\rho}))
\]
together with a representation 
\[
\rho^{(p)}: \Gal(\overline{\Q}/\Q) \to \SL_2(\Z_3)
\]
such that
\begin{enumerate}
    \item $\rho^{(p)} \equiv \rho \bmod 3$;
    \item $\rho^{(p)}$ is ramified at $p$ and unramified outside $S\cup \{p\}$;
    \item $\rho^{(p)}|_{G_v}$ factors through $\Z_3[[T_1,\dots,T_{n_v}]]$
for all $v\in S \cup \{p \}$. 
\end{enumerate}
\noindent In line with the previous notation, let  $\mathcal{C}_p$ denote the class of local deformations of $\overline{\rho}|_{G_p}$ that factor through the smooth quotient \eqref{smoothp} of $R_p$.
%Finally, for $T= S$ or $T=S\cup \{p\}$,
%let $R_{(\mathcal{N}_v)_{v \in T}}$ be the global deformation ring parametrizing lifts $\varrho$ of $\overline{\rho}$
%such that $\varrho|_{G_v} \in \mathcal{C}_v$ at all $v\in T$.

\begin{proof}[Proof of Corollary \ref{selmerdensity}]
Let  $\infty$ denote the Archimedean place of $\Q$ and let
$
S=\{\ell, 3, \infty\}.
$
It follows from the results in \cite{even1}
that
for each $v\in S$, the local deformation ring $R_v$ has a smooth quotient
\begin{equation} \label{smooth}
R_v \to \Z_3[[T_1,\dots,T_{n_v}]]
\end{equation}
with tangent space $\mathcal{N}_v \subseteq H^1(G_v, \operatorname{Ad}^0(\overline{\rho}))$, where  
\[
\mathcal{N}_\ell = H^1_{\operatorname{unr}}(G_\ell, \operatorname{Ad}^0(\overline{\rho})), \quad  
\mathcal{N}_3 = H^1(G_3, \operatorname{Ad}^0(\overline{\rho})), \quad \mathcal{N}_\infty = 0,
\]
for which 
the global setting is balanced of rank zero at $S$.\\

\noindent
For $p \in \mathcal{C}$,
let
$\sigma_p \in G_p$ be a lift of the Frobenius automorphism and $\tau_p \in G_p$ a generator of inertia.
Let $\mathcal{C}_p$ 
be the class of 
$\overline{\rho}|_{G_p}$-deformations 
$\varrho: G_p \to \SL_2(A)$ for Artin algebras $A$ over $\Z_3$
such that
\begin{equation}\label{newatp}
\varrho(\sigma_p) = 
\begin{pmatrix}
    p^{1/2} & 1 + x \\ & p^{-1/2}
\end{pmatrix}, 
\quad
\varrho(\tau_p) = \begin{pmatrix}
    1 &  y \\ & 1
\end{pmatrix}
\end{equation}
for some $x,y$ in the maximal ideal of $A$.
Then the local deformation ring $R_p$ has a quotient
isomorphic to a power series ring  $\Z_3[[T_1,\ldots, T_{n_p}]]$
with 
$n_p = \dim H^0(G_p, \operatorname{Ad}^0(\overline{\rho}))$ 
such that 
$\mathcal{C}_p$ is the class of deformations 
that factor through $\Z_3[[T_1,\ldots,T_{n_p}]]$.
By an application of Wiles' formula, we conclude that  
the global setting remains balanced after allowing ramification at $p$: 
\[
  \dim H^1_{\mathcal{N}}(G_{S\cup \{ p \}}, \operatorname{Ad}^0(\overline{\rho}) )
  = 
  \dim H^1_{\mathcal{N}^\perp}(G_{S\cup \{ p \}}, \operatorname{Ad}^0(\overline{\rho})^* ).
\]
Let $f^{(p)}$ be the unique global cohomology class in $H^1(G_{S\cup \{p\} }, \operatorname{Ad}^0(\overline{\rho}))$ that is ramified at $p$ and unramified at $\ell$. Then 
for all $v\in S=\{\ell, 3 ,\infty\}$, we have
$f^{(p)}|_{G_v} \in \mathcal{N}_v$.
Moreover, $f^{(p)}|_{G_p} \notin \mathcal{N}_p$ if and only if 
$f(p, K^{(p)}/\Q) = 9$ if and only if 
\[
\dim H^1_{\mathcal{N}}(G_{S\cup \{ p \} }, \operatorname{Ad}^0(\overline{\rho}))=
\dim H^1_{\mathcal{N}}(G_{S }, \operatorname{Ad}^0(\overline{\rho})),
\]
proving Corollary \ref{selmerdensity}.
\end{proof}

\bibliographystyle{amsplain} 
\bibliography{references.bib}
\end{document}